\documentclass[10pt]{article}
\usepackage{graphicx}
\usepackage{amsfonts}
\usepackage{amsmath}
\usepackage{amsthm}
\usepackage{amssymb}
\usepackage{hyperref}
\usepackage{mathtools}
\usepackage{geometry}
\usepackage{xcolor}
\usepackage{comment}
\usepackage{verbatim}
\usepackage{authblk}
\usepackage{todonotes}

\definecolor{CGcol}{rgb}{1.,0.,0.}

\newcommand{\dee}{{\rm d}}

\newcommand{\Vee}{\mathrm{V}}
\DeclareMathOperator{\Div}{div}

\DeclareMathOperator{\Id}{Id}
\newcommand{\R}{\mathbb{R}}
\newcommand{\w}{\mathrm{w}}

\newcommand{\wmdw}{\partial^\bullet_\w}
\newcommand{\mdp}[1]{\overset{\bullet \phi}{#1}}
\newcommand{\mdplong}[1]{\overbrace{#1}^{\bullet \phi}}
\newcommand{\wmdp}{\partial^\bullet_\phi}
\newcommand{\VO}{\mathcal{V}}
\newcommand{\HO}{\mathcal{H}}

\newcommand{\changed}[1]{{#1}}

\newtheorem{theorem}{Theorem}[section]
\newtheorem{corollary}[theorem]{Corollary}
\newtheorem{lemma}[theorem]{Lemma}
\newtheorem{proposition}[theorem]{Proposition}
\newtheorem{definition}[theorem]{Definition}
\newtheorem{remark}[theorem]{Remark}
\newtheorem{problem}[theorem]{Problem}
\newtheorem{assumption}[theorem]{Assumption}

\title{An evolving space framework for Oseen equations on a moving domain}
\date{\today}

\newcommand{\thanksAD}{
    This research has been
    funded %
    by Deutsche Forschungsgemeinschaft (DFG)
    through the grant CRC 1114: “Scaling Cascades in Complex Systems”,
    Project Number 235221301
}
\newcommand{\thanksPH}{
    P.J.H. acknowledges the support of EPSRC (grant EP/W005840/1).
}
\newcommand{\thanksCG}{
    This research has been
    supported %
    by Deutsche Forschungsgemeinschaft (DFG)
    through the grant CRC 1114: “Scaling Cascades in Complex Systems”,
    Project Number 235221301
}

\author[1]{Ana Djurdjevac\thanks{\thanksAD}}
\affil[1]{Freie Universit\"at Berlin,
  Institut f\"ur Mathematik,
  Arnimallee 6, 14195 Berlin}

\author[2]{Carsten Gr\"aser\thanks{\thanksCG}}
\affil[2]{
  Friedrich-Alexander-Universit\"at Erlangen-N\"urnberg,
  Department of Mathematics,
  Cauerstraße 11, 91058 Erlangen
}

\author[3]{Philip J. Herbert\thanks{\thanksPH}}

\affil[3]{Department of Mathematics, University of Sussex, Brighton, BN1 9RF, United Kingdom
}

\affil[ ]{adjurdjevac@zedat.fu-berlin.de, graeser@math.fau.de, p.herbert@sussex.ac.uk}

\begin{document}
\maketitle
\begin{abstract}
This article considers non-stationary incompressible linear fluid equations in a moving domain.
We demonstrate the existence and uniqueness of an appropriate weak formulation of the problem by making use of the theory of time-dependent Bochner spaces.
It is not possible to directly apply established evolving Hilbert space theory due to the incompressibility constraint.
After we have established the well-posedness, we derive and analyse a first order time discretisation of the system.
\end{abstract}

\section{Introduction}
The study of Partial Differential Equations (PDEs) in or on moving domains has received much attention recently, particularly the case of moving surfaces.
This is not only due to the interesting analysis required but also the range of applications where it is necessary to consider the motion of a domain for a more accurate model.
Typical applications include biological and physical phenomena, including pattern formation \cite{VenTosEam11,BarEllMad11, EllSti10A, EllSti10B} and the
modelling of surfactants in multi-phase flow \cite{GarLamSti14,DunLamSti19}.

Classically the study of PDEs on moving domains involves pulling the equations back onto a given stationary reference domain.
The work \cite{Vie14} introduced a moving space framework whereby the equations are posed in time-dependent Bochner spaces.
Shortly after, \cite{AlpEllSti15A} extended this to general Hilbert spaces, with examples given in \cite{AlpEllSti15B} and to Banach spaces in \cite{ACDJE21}.
Non-linear extensions have been considered in \cite{AlpEllTer18,AlpEll15} and non-local in \cite{AlpEll16}.
These examples and extensions however do not cover incompressible fluid equations.

We are motivated by the many physical situations in which one finds a moving domain in a fluid, examples range from large scale engineering
	 \cite{heinrich1992nonlinear},  computational geophysics, 
  in-cylinder flows in
internal  engines to  cardiovascular biomechanics \cite{lee2013moving,liu2020fluid, quarteroni2019mathematical}, etc.
For more examples see for instance \cite[Ch. 1]{shyy1995computational}.
Study of fluid equations in moving domains has been mathematically considered.
Some of the classical results are \cite{MiyTer82, InoWak77, bock1977navier, fujita1970existence}.
Those have been developed further by numerous authors such as \cite{moubachir2006moving, Saa06, Saa07,MiyTer82,InoWak77,FujSau70,Sal88}.
Note however that these approaches are based on the so-called method of pulling-back to the fixed domain.
The aim of this article is to utilise time-dependent Bochner spaces to provide a structure which can directly be used for proving the well-posedness results, without the need to transform to a fixed domain.
For the equations we consider, it is possible to pull back the PDE system to a reference space and apply standard theory on fixed domains after verifying the appropriate assumptions.
The approach we develop allows the problem to be considered in its natural formulation.
Handling the equations in natural spaces gives an elegance and simplicity which becomes particularly apparent in numerical implementation and finite element analysis \cite{EllRan20} of PDEs on evolving surfaces/domains, while also being an interesting mathematical problem.

This article extends the moving space framework of \emph{Alphonse, Elliott} and \emph{Stinner} to include the example of incompressible fluid equations.
We focus on a linear version of the Navier-Stokes Equations in a flat domain.
We will derive and analyse a first order time-discretisation of the equation we study.
We neglect the spatial discretisation.
However, we will design the time-stepping method such that
it is suited to an evolving finite element method,
   which relies on moving the mesh and basis functions
   together with the domain.
   In each time step this could be based on a classical
mixed Finite Element Method \cite{BofBreFor13}.

As commented in \cite{Dju18}, the challenging aspects of the incompressible fluid equations lie in the constraint that the velocity should remain divergence free.
If one uses a standard pullback to the reference domain, the divergence free constraint on the velocity becomes a differential constraint which depends on time.
In order to preserve the structure of the divergence constraint under domain deformations, one can use the so-called contravariant Piola transform.
This appears in \cite{MiyTer82}, \cite{bock1977navier}, and \cite{BurFreMas19}.
Both of these approaches have the difficulty that, instead of a moving domain, one has time dependent coefficients \changed{(c.f. \cite{bock1977navier})}.
This can prove challenging when ensuring appropriate approximation properties in a spatial discretisation.

We will use this contravariant Piola transform.
As a consequence of using a different pullback map, many of the necessary assumptions from \cite{AlpEllSti15A} need to be verified.
In addition, we consider a different pivot space from the typical choice of $L^2$, this means that there are some technical challenges which need to be addressed.

Discretisations for fluid equations in moving domains have been considered as early as the 1990s, where \cite{demirdvzic1990finite} considers a finite volume discretisation for Navier-Stokes in a moving domain.
The article \cite{tezduyar2004finite} provides an overview of of methods for discretisation of fluid dynamics with evolving domains.
Let us note \cite{BurFreMas19}, which considers a method for discretising a non-stationary Stokes in a moving domain using geometrically unfitted meshes and ghost penalties.

Voulis and Reusken \cite{VouReu18} consider an interface problem for a non-stationary Stokes equation, they show well-posedness and give space-time finite element discretisations.
We note that, while they have an evolving interface, their fluid is contained within a stationary domain, therefore they do not need to use evolving space methods.
It is worth commenting that the problem we consider may be considered as one of the phases of their interface problem.

Many of the evolving function space problems previously mentioned are for parabolic equations on surfaces.
We do not currently consider surface problems for the fluid equations, see \cite{KobLiuGig17,Miu17,JanOlsReu18,BraReuSch22} for such models.
We point out the article \cite{OlsReuZhi22} which considers the analysis
of a tangential Navier--Stokes system using a
similar evolving space methodology to that which we consider.
The primal evolving spaces are also defined using the
contravariant Piola transformation.
Notice, however, that this work differs from the presented one in three fundamental aspects:
First, \cite{OlsReuZhi22} considers 2-dimensional closed surfaces which does
not cover the case of flat 2- and 3-dimensional domains with (non-periodic)
  boundary conditions.
  Second, the dual evolving space of the Gelfand-triple
  in \cite{OlsReuZhi22} is defined using the adjoint (covariant) of the
  contravariant Piola transform, while we use the contravariant Piola
  itself which simplifies computations.
  Finally, \cite{OlsReuZhi22} proves well-posedness using a discretisation approach,
  whereas we apply the Banach-Ne\v cas-Babu\v ska Theorem.

It is furthermore worth noting that fluid equations on stationary hypersurfaces have been the subject of recent study in particular the numerical analysis of a Navier-Stokes equation \cite{ReuVoi18} and Stokes equations \cite{Reu20,OlsQuaReu18}.

\subsection{Formulation of the problem}
We are interested in the flow of an incompressible fluid in a moving domain.
For our purposes the flow of a fluid will be governed by a parabolic Oseen equation.
The parabolic Oseen equation is a linear version of the Navier--Stokes equations.
In the future, one may wish to consider the non-linear Navier--Stokes equations in this moving domain setting.

For simplicity we will consider that the moving domain is a subset of an open bounded hold-all domain in $\R^d$, $d=2,3$.
We will denote our moving domain by $\Omega(t)$, $t \in [0,T]$ for some $T>0$, assuming that it flows with a given, sufficiently smooth velocity field $\w$, which is divergence free.
The standard parabolic Stokes equations in this domain moving with $\w$ would be given as: find $(u,p)$ such that
\begin{equation}\label{eq:materialVelocityWSimpleEquation}
	\wmdw u - \mu \Delta u + \nabla p = 0,\, \Div u = 0 \mbox{ in }\Omega(t),
\end{equation}
where $\wmdw$ is a physical material derivative defined by $\wmdw u := u_t + (\w \cdot \nabla) u$ and corresponds to the movement of the domain induced by $\w$ and $\mu>0$ is the viscosity.
For simplicity, we will consider the case $\mu =1$, the analysis we present will be valid for $\mu\neq 1$, however the constants which appear will depend on $\mu$.
We will consider the more general case given by the parabolic Oseen equation: given $\Vee$, find $(u,p)$ such that
\begin{equation}\label{eq:ParabolicOseen}
	u_t + (\Vee\cdot\nabla)u - \Delta u + \nabla p = 0,\, \Div u = 0 \mbox{ in } \Omega(t),
\end{equation}
where, for $\Vee= \w$, one recovers the previous equation \eqref{eq:materialVelocityWSimpleEquation}; also see that for unknown $\Vee = u$, one is in the non-linear Navier--Stokes setting.
We note that one may also consider a porous medium term $c u$, where $c$ is a sufficiently regular scalar function and $c\geq 0$ almost everywhere.

Regarding the boundary conditions, we will consider non-zero Dirichlet boundary conditions, assuming that $u(t,\cdot)|_{\partial\Omega(t)} = u_d(t,\cdot) |_{\partial\Omega(t)}$, where $u_d$ is a given function which is divergence free and sufficiently smooth, i.e. it belongs to $H^1$.
Taking the particular choice $u_d  = \w$ effectively corresponds to a no-slip condition, other conditions are possible \cite{Mon13}.
We note that although the geometry of $\Omega(t)$ is determined entirely by the normal component of $\w$ on $\partial\Omega(t)$, when considering the no-slip condition the tangential component of $\w$ will affect the solution $u$.
This is exemplified by the Taylor-Couette flow between two concentric cylinders \cite{Tay23}.

We now introduce the problem we wish to study:

\begin{problem}\label{prob:NonZeroDirichlet}
	Given $\hat f$, $\hat u_0$, and $\Vee$, find velocity field $\hat u$ and pressure field $p$ such that
	\begin{align}
		\hat u_t + (\Vee \cdot \nabla) \hat u - \Delta \hat u + \nabla p &= \hat f \quad &\mbox{in }\cup_{t\in(0,T)}\{t\} \times \Omega(t),
		\\
		\Div \hat u &=0 \quad &\mbox{in }\cup_{t\in(0,T)}\{t\} \times \Omega(t),
		\\
		\hat u&= u_d \quad &\mbox{on } \cup_{t\in(0,T)}\{t\} \times \partial \Omega(t),
		\\
		\hat u|_{t=0} &= \hat u_0 \quad &\mbox{on } \{0\} \times \Omega(0).
	\end{align}
\end{problem}
In order to obtain zero Dirichlet boundary conditions, we consider $u := \hat{u}-u_d$, the above problem then becomes:
\begin{problem}\label{prob:ZeroDirichlet}
	Given $\hat f$, $\hat u_0$, and $\Vee$ find velocity field $u$ and pressure field $p$ such that
	\begin{align}
		u_t + (\Vee \cdot \nabla) u - \Delta u + \nabla p &= f: = \hat{f} - (u_d)_t - (\Vee \cdot \nabla )u_d + \Delta u_d \quad &\mbox{in }\cup_{t\in(0,T)}\{t\} \times \Omega(t),
		\\
		\Div u &=0 \quad &\mbox{in }\cup_{t\in(0,T)}\{t\} \times \Omega(t),
		\\
		u &= 0 \quad &\mbox{on } \cup_{t\in(0,T)}\{t\} \times \partial \Omega(t),
		\\
		u|_{t=0} &= u_0 := \hat u_0- u_d|_{t=0} \quad &\mbox{on } \{0\} \times \Omega(0).
	\end{align}
\end{problem}
Our goal in this article is to apply, where possible, the methods of \cite{AlpEllSti15A} to this system by considering the certain time-dependent Bochner spaces along with an appropriate transformation.

\subsection{Outline}

The structure of the paper is the following. 
We begin in Section \ref{sec:PreliminariesAssumptionsConsequences} by giving the conditions we require on the movement of our domain and briefly  introducing the  flow transformation of the domain.
This is followed by the construction of the evolving space structure in Section \ref{sec:EvolvingFunctionSpaces}. More precisely,  we introduce divergence free function spaces and the so-called Piola transformation which has the property of preserving the divergence-free structure.  Based on this setting and the approach from \cite{AlpEllSti15A},  in Section \ref{sec:Time-dependentBocnherSpaces} we define time-dependent Bochner spaces and material derivative (Section \ref{MaterialDerivative}), which are exploited in Section \ref{SolutionSpace} for the definition of the solution space.
Existence and uniqueness,  in a suitably weak sense, of Problem \ref{prob:NonZeroDirichlet}
are proved in Section \ref{sec:WellPosedness}.
We state and derive time discretisation strategy for Problem \ref{prob:NonZeroDirichlet}
in Section \ref{sec:TimeDiscretisation} for which we also prove convergence.

\section{Incompressible function spaces on moving domains}

\subsection{Assumptions on the evolution of the domain}

\label{sec:PreliminariesAssumptionsConsequences}

We now outline some requirements we have on the initial domain and domain velocity $\w$, these conditions are used in order to properly define the evolving Hilbert space structures which we will use.
The conditions we require for $\Vee$ are deferred until Assumption \ref{ass:AssumptionsForV}.
We will henceforth fix $D\subset \R^d$ a bounded open domain, $\Omega_0$ open and compactly contained in $D$ with $C^2$ boundary, and $\w \colon D \to \R^d$.

\begin{assumption}\label{ass:AssumptionsForW}
	Assume that there is $k\geq 2$ such that $\w \in C^1\left(\overline{(0,T)} ;C^k_c(D;\R^d)\right)$ with compact support in space.
	Also assume that $\Div \w = 0$ in $[0,T] \times D$.
\end{assumption}

We now define the transformation induced by $\w$.
\begin{definition}
	Let $\Phi\colon [0,T] \times D \to D$ be the solution of the ODE
	\begin{align}
		\frac{\dee}{\dee t}\Phi(t,\cdot) =& \w(t,\Phi(t,\cdot))& &\mbox{ in } D, \mbox{ for } t \in (0,T),
		\\
		\Phi(0,\cdot) =& {\rm Id}_D& &\mbox{ in } D.
	\end{align}
	We write $\Phi_t:= \Phi(t,\cdot)$.
\end{definition}
The fact that this $\Phi$ exists is a standard result in the theory of ODEs, see \cite{Har02} for example.
\begin{lemma}\label{lem:RegularityPhi}
	The map $\Phi\colon [0,T]\times D \to D$ exists, is unique and is a $C^k$-diffeomorphism.
\end{lemma}
We now use this flow map to define our moving domain.
\begin{definition}
	We define the family of moving domains $\{\Omega(t)\}_{t \in [0,T]}$ by $\Omega(t):= \Phi_t(\Omega_0)$, the image of $\Omega_0$ under $\Phi_t$ for each $t \in [0,T]$.
\end{definition}
We now have the following properties of the family $\{\Omega(t)\}_{t \in (0,T)}$.
\begin{lemma}\label{lem:propertiesOfDomainAndDet}
	It holds that $\Omega(0) = \Omega_0$.
	For each $t \in [0,T]$, $\Omega(t)\subset D$ and has $C^2$ boundary.
	Finally, $\det(D\Phi_t) = 1 $ in $\Omega_0$ for all $t\in [0,T]$.
\end{lemma}
\begin{proof}
    First note that since $\Phi_t = {\rm Id}_D$, we have that $\Omega(0) = \Omega_0$
    The fact that $\Omega(t) \subset D$ follows from the $\w$ having compact support in $D$.
    Furthermore, since $\Phi_t$ is a $C^2$-diffeomorphism, $\partial \Omega(t)$ is $C^2$.
    Finally, utilising $\Div w =0$ and $\Phi_0 = {\rm Id}_D$, one has $\det(D\Phi_t) = 1$.
	For further details, see Proposition 1.4 of \cite{MajBer01}.
\end{proof}

\subsection{Evolving function spaces and compatibility}
\label{sec:EvolvingFunctionSpaces}

We begin by defining function spaces of interest.
\begin{definition}
Let
\begin{align}\label{def:spaces}
\VO :=& \{ u \in H_0^1(\Omega_0; \mathbb{R}^d): \Div u = 0\}, \\
\HO :=& \{ u \in L^2(\Omega_0; \mathbb{R}^d): u \mbox{ is weakly divergence free}\},
\end{align}
where we say $u \in H^{-1}(\Omega_0;\R^d)$ is weakly divergence free if
\begin{equation}
	\langle u, \nabla q\rangle_{H^{-1}, H^1} = 0\quad \forall q \in C_0^2(\Omega_0;\R).
\end{equation}
We will equip $\VO$ with the $H^1(\Omega_0)$ inner product, and $\HO$ with the $L^2(\Omega_0)$ inner product:
\begin{equation}
	(u,v)_{H^1}= \int_{\Omega_0} Du : Dv + u \cdot v, \qquad
	(u,v)_{L^2}= \int_{\Omega_0} u\cdot v.
\end{equation}

The time-dependent spaces are, for each $t \in [0,T]$,
\begin{align}
V(t) :=& \{ \tilde{u} \in H_0^1(\Omega(t); \mathbb{R}^d): \Div \tilde{u} = 0\}, \\
H(t) :=& \{ \tilde{u} \in L^2(\Omega(t); \mathbb{R}^d): \tilde{u} \mbox{ is weakly divergence free}\},
\end{align}
where we say for $\tilde{u} \in H^{-1}(\Omega(t);\R^d)$, is weakly divergence free if
\begin{equation}\label{eq:WeaklyDivergenceFree}
	\langle \tilde{u}, \nabla \tilde{q}\rangle_{H^{-1},H^1} = 0 \quad \forall \tilde{q} \in C_0^2(\Omega(t);\R).
\end{equation}
We will equip $V(t)$ with the $H^1(\Omega(t))$ inner product, and $H(t)$ with the $L^2(\Omega(t))$ inner product:
\begin{equation}
	(\tilde{u},\tilde{v})_{H^1}:= \int_{\Omega(t)} D\tilde{u} : D\tilde{v} + \tilde{u} \cdot \tilde{v}, \qquad
	(\tilde{u},\tilde{v})_{L^2}:= \int_{\Omega(t)} \tilde{u}\cdot \tilde{v}.
\end{equation}
\end{definition}
Throughout, on inner-product spaces, we will use the induced norms $\|u\|_X = \sqrt{(u,u)_X}$.
In the above, we have dropped the domain and codomain for the spaces which appear in the subscripts to avoid cumbersome notation, this will be done whenever it is clear as to which is the appropriate domain.

It is clear that these spaces are Hilbert with their inner products.
We also have that
\begin{equation}
V(t) \subset H(t) \cong H(t)^* \subset V^*(t)
\end{equation}
is a Gelfand triple for every $t \in [0,T]$.
The evolving family of Hilbert spaces that we will consider are
$\{V(t)\}_{t\in(0,T)}$ and $\{ H(t) \}_{t\in(0,T)}$.
In this setting, we make use of the characterisation of the dual space \cite[Page 8]{Tem87}
\begin{equation}
	V^*(t) = \{ \tilde{u}\in H^{-1}(\Omega(t);\R^d) : \tilde{u} \mbox{ is weakly divergence free} \}.
\end{equation}
We will repeatedly use the fact that $V(t) \subset H^1_0(\Omega(t);\R^d)$, $H(t) \subset L^2(\Omega(t);\R^d)$, and $V^*(t) \subset H^{-1}(\Omega(t);\R^d)$.

We now define a family of maps which will transform functions on $\Omega_0$ to functions on $\Omega(t)$.
\begin{definition}[Piola transform]\label{eq:contravariantPiola}
	For each $t \in [0,T]$, the linear map $\phi_t\colon L^2(\Omega_0;\R^d) \to L^2(\Omega(t);\R^d)$ is defined by
	\begin{equation}
		\phi_t u := \left( D\Phi_t u \right) \circ \Phi_t^{-1}\label{eq:contravariantPiola_1}
	\end{equation}
	for $u \in L^2(\Omega_0;\R^d)$.
	Similarly define the family of linear maps $(\phi_{-t})_{t \in [0,T]} \colon L^2(\Omega(t);\R^d) \to L^2(\Omega_0;\R^d)$ by
	\begin{equation}
		\phi_{-t} \tilde{u} := D\Phi_t^{-1} \left(\tilde{u}\circ \Phi_t\right)\label{eq:contravariantPiola_2}
	\end{equation}
	for $\tilde{u} \in L^2(\Omega(t);\R^d)$, for each $t \in [0,T]$.
\end{definition}
Notice that this transformation is different to the plain pull-back considered in \cite{AlpEllSti15A,AlpEllSti15B}, which is induced by the given diffeomorphism of the domains.
The reason for using this map is that we want to ensure that the transformation takes divergence free functions to divergence free functions.
This transformation is known as the (contravariant) Piola transform in the particular case that $\det(D\Phi_t)$ is spacially constant, for more details see \cite{RogKirLog10}, for example.
With this definition in mind and the appropriate regularity of $\Phi_t$, courtesy of Lemma \ref{lem:RegularityPhi}, we give the result that $\phi_t$ and $\phi_{-t}$ preserve the divergence of the functions they are applied to.

\begin{lemma}\label{lem:divPreserving}
	For each $t \in (0,T)$, let $u \in C^1(\Omega_0;\R^d)$, it holds that
	\begin{equation}
		\Div(\phi_t u) = \left(\Div u\right) \circ \Phi_t^{-1},
	\end{equation}
	in addition, for $\tilde{u}\in C^1(\Omega(t):\R^d)$ it also holds that
	\begin{equation}
		\Div (\phi_{-t}\tilde{u}) = \left(\Div \tilde{u}\right) \circ \Phi_t.
	\end{equation}
\end{lemma}
This result may be found in e.g.\ \cite{RogKirLog10} and makes use of Jacobi's formula for the derivative of a determinant.
For completeness, we provide the proof in Appendix \ref{appendix:ProofOfLemma}.

We now state that $\phi_t$ and $\phi_{-t}$ are bounded in appropriate norms, the proof of which is also in Appendix \ref{appendix:ProofOfLemma}.
\begin{lemma}\label{lem:phiBoundedWithInverse}
	For any $t \in (0,T)$, there are constants $C_1, C_2, C_3, C_4 >0$ independent of $t$ such that
	\begin{align}
		\|\phi_t u \|_{H^1(\Omega(t))} \leq& C_1 \|u \|_{H^1(\Omega_0)} \quad \forall u \in H^1(\Omega_0;\R^d),
		\\
		\|\phi_t u \|_{L^2(\Omega(t))} \leq& C_2\|u \|_{L^2(\Omega_0)} \quad \forall u \in L^2(\Omega_0;\R^d),
		\\
		\|\phi_{-t} \tilde{u} \|_{H^1(\Omega_0)} \leq& C_3 \|\tilde{u} \|_{H^1(\Omega(t))} \quad \forall \tilde{u} \in H^1(\Omega(t);\R^d),
		\\
		\|\phi_{-t} \tilde{u} \|_{L^2(\Omega_0)} \leq& C_4 \|\tilde{u} \|_{L^2(\Omega(t))} \quad \forall \tilde{u} \in L^2(\Omega(t);\R^d).
	\end{align}
\end{lemma}

We now define the dual operators to $\phi_t$ and $\phi_{-t}$, this will allow us to appropriately define the extension of $\phi_t$ to $H^{-1}(\Omega_0;\R^d)$, similarly for $\phi_{-t}$ to $H^{-1}(\Omega(t);\R^d)$.

\begin{definition}\label{def:covariantPiola_1}
	For each $t \in [0,T]$, we define the linear maps
	$\phi_t^* \colon L^2(\Omega(t);\R^d) \to L^2(\Omega_0;\R^d)$ and
	$\phi_{-t}^* \colon L^2(\Omega_0;\R^d) \to L^2(\Omega(t);\R^d)$ by
	\begin{align}
	\phi_t^* \tilde{u}:=& D\Phi_t^{T} \left(\tilde{u}\circ \Phi_t\right),\label{eq:covariantPiola_1}
	\\
	\phi_{-t}^* u :=& \left( D\Phi_t^{-T} u \right) \circ \Phi_t^{-1},\label{eq:covariantPiola_2}
	\end{align}
	for $\tilde{u} \in L^2(\Omega(t);\R^d)$ and $u \in L^2(\Omega_0;\R^d)$.
\end{definition}
These maps are known as the covariant Piola transforms \cite{RogKirLog10}.
In future appearances of composition with maps, we will relax the number of brackets to avoid an excessive number in a single expression e.g.~for \eqref{eq:contravariantPiola_2}, we will write $D\Phi_t^{-1} \tilde{u} \circ \Phi_t$
It is always the case that only the final term should be composed with the map (with brackets).

\begin{remark}
\hspace{1mm}
\begin{itemize}
\item
By a minor modification to the proofs of Lemma \ref{lem:phiBoundedWithInverse}, it follows that the maps $\phi_t^*$ and $\phi_{-t}^*$ are bounded.
\item
It is also possible to see that $\phi_t^*$ is dual to $\phi_t$ in the following sense:
	for $u \in L^2(\Omega_0)$, $\tilde{\eta} \in L^2(\Omega(t))$, we calculate, utilising that $\det(D\Phi_t) = 1$,
	\begin{equation}\begin{split}\label{eq:dualityCalculationPt1}
		(\phi_t u, \tilde{\eta})_{L^2}
		=
		\int_{\Omega(t)} \phi_t u \cdot \tilde{\eta}
		&=
		\int_{\Omega(t)} \left( D\Phi_t u\right) \circ \Phi_t^{-1} \cdot \tilde{\eta}
		\\
		&=
		\int_{\Omega_0} u \cdot  D\Phi_t^T \tilde{\eta}\circ \Phi_t
		=
		\int_{\Omega_0} u \cdot \phi_t^* \tilde{\eta}
		=
		(u,\phi_t^* \tilde{\eta})_{L^2}.
	\end{split}\end{equation}
	We may do a similar calculation for $\phi_{-t}^*$ and $\phi_{-t}$ and $\tilde{u} \in L^2(\Omega(t))$ and $\eta \in L^2(\Omega_0)$,
	\begin{equation}\begin{split}\label{eq:dualityCalculationPt2}
		(\phi_{-t} \tilde{u},\eta)_{L^2}
		=
		\int_{\Omega_0} \phi_{-t} \tilde{u} \cdot \eta
		&=
		\int_{\Omega_0} D\Phi_t^{-1} \tilde{u} \circ \Phi_t \cdot \eta
		\\
		&=
		\int_{\Omega(t)} \tilde{u} \cdot \left( D\Phi_t^{-T} \eta \right) \circ \Phi_t^{-1}
		=
		\int_{\Omega(t)} \tilde{u} \cdot \phi_{-t}^* \eta
		=
		(\tilde{u},\phi_{-t}^*\eta)_{L^2}.
	\end{split}\end{equation}
\end{itemize}
\end{remark}

An immediate consequence of this remark is that the maps $\phi_t$ and $\phi_{-t}$ may be extended to distributions via duality.
\begin{definition}\label{def:DualSpaceMapDefintion}
	For each $t \in [0,T]$, we define the map $\phi_t \colon H^{-1}(\Omega_0;\R^d) \to H^{-1}(\Omega(t);\R^d)$ by: for $u \in H^{-1}(\Omega_0;\R^d)$, let $\phi_t u$ to be the unique element $g \in H^{-1}(\Omega(t);\R^d)$ such that
	\begin{equation}
		\langle g, v\rangle_{H^{-1}, H^1}
		=
		\langle u, \phi_t^* v\rangle_{H^{-1}, H^1}\quad \forall v \in H^1_0(\Omega(t);\R^d).
	\end{equation}
	For each $t \in [0,T]$, We define the map $\phi_{-t} \colon H^{-1}(\Omega(t);\R^d) \to H^{-1}(\Omega_0;\R^d)$ by: for $\tilde{u} \in H^{-1}(\Omega(t);\R^d)$, we define $\phi_{-t} \tilde{u}$ to be the unique element $\tilde{g} \in H^{-1}(\Omega_0;\R^d)$ such that
	\begin{equation}
		\langle \tilde{g}, \tilde{v}\rangle_{H^{-1}, H^1}
		=
		\langle \tilde{u}, \phi_{-t}^* \tilde{v}\rangle_{H^{-1}, H^1} \quad \forall \tilde{v}\in H^1_0(\Omega_0;\R^d).
	\end{equation}
\end{definition}
Notice that, by the calculations in \eqref{eq:dualityCalculationPt1} and \eqref{eq:dualityCalculationPt2}, this is seen to be an extension of $\phi_t$ and $\phi_{-t}$ as defined on $L^2(\Omega_0)$ and $L^2(\Omega(t))$ respectively.
The unique elements $g$ and $\tilde{g}$ exist by applications of the Riesz representation theorem.
We comment that this definition of the maps between elements of $H^{-1}(\Omega(t);\R^d)$ and $H^{-1}(\Omega_0;\R^d)$ are in the same spirit of \cite{AlpEllSti15A}.
However, the maps are considered differently due to the way we wish to characterise our spaces; in particular the fact that we wish for our $V^*$ spaces to be divergence free $H^{-1}$ spaces, rather than the abstract dual of the $V$ spaces, which will be larger than $H^{-1}$.
\begin{lemma}\label{lem:ExtendedMapVStarToVStarT}
    For $u \in \VO^*$, it holds that $\phi_t u \in V^*(t)$ and for $\tilde{u} \in V^*(t)$ it holds that $\phi_{-t}\tilde{u} \in \VO^*$.
\end{lemma}
\begin{proof}
    The result follows by seeing that when $\eta \in C_0^2(\Omega_0)$, $\nabla (\eta \circ \Phi_t^{-1}) = \phi_{-t}^* \nabla \eta$, therefore
    \begin{equation}
        \langle \phi_{-t}\tilde{u},\nabla \eta\rangle_{H^{-1}, H^1}
        =
        \langle \tilde{u}, \phi_{-t}^* \nabla \eta\rangle_{H^{-1}, H^1}
        =
        \langle \tilde{u}, \nabla (\eta \circ \Phi^{-1}) \rangle_{H^{-1}, H^1} = 0,
    \end{equation}
    since $\tilde{u}$ is weakly divergence free and $\eta \circ \Phi_t^{-1} \in C_0^2(\Omega(t))$.
    The converse direction follows through the same idea.
\end{proof}

\begin{lemma}\label{lem:BoundednessForMapOnDual}
	For any $t\in [0,T]$, there are $C_1,\,C_2>0$ independent of $t$ such that
	\begin{align}
	\|\phi_t u \|_{H^{-1}(\Omega(t)) } \leq& C_1 \|u \|_{H^{-1}(\Omega_0)} \quad \forall u \in H^{-1}(\Omega_0;\R^d),
	\\
	\|\phi_{-t} \tilde{u} \|_{H^{-1}(\Omega_0)} \leq& C_2 \|\tilde{u} \|_{H^{-1}(\Omega(t))} \quad \forall \tilde{u} \in H^{-1}(\Omega(t);\R^d).
	\end{align}
\end{lemma}
\begin{proof}
This is a consequence of using the dual norm, the boundedness of $\phi_t^*$ and $\phi_{-t}^*$ and that the transformation is given by duality.
\end{proof}

We now define what it means for our moving spaces to be \emph{compatible} in the sense of \cite{AlpEllSti15A}.
\begin{definition}
A pair $(X, (\phi_t)_t)$ is compatible if and only if the following holds:
\begin{itemize}
\item for every $t \in [0,T]$, $X(t)$ is a real separable Hilbert space and the map $\phi_t:X_0 \to X(t)$ is a linear homeomorphism such that $\phi_0$ is the identity.

\item there exists a constant $C_X$ independent of $t$ such that
\begin{align}
\| \phi_t u \|_{X(t)} \leq& C_X \| u \|_{X_0}& &\forall u \in X_0 \\
\| \phi_{-t} u \|_{X_0} \leq& C_X \| u \|_{X(t)}& &\forall u \in X(t)
\end{align}
where $\phi_{-t}\colon X(t) \to X_0$ is the inverse of $\phi_t$.
\item the map $t \mapsto \| \phi_t u \|_{X(t)} $is continuous for all $u \in X_0$.
\end{itemize}
\end{definition}

\begin{proposition}\label{prop:wholeSpaceCompatible}
	The pairs $(H^1, \phi)$, $(L^2,\phi)$, and $(H^{-1},\phi)$ are compatible pairs.
\end{proposition}
The proof of this follows as in the following result.
\begin{proposition}\label{prop:compatibility}
The pairs $(V^*,\phi)$, $(H, \phi)$ and $(V, \phi)$ are compatible.
\end{proposition}
\begin{proof}
	It is standard to verify that $V^*(t)$, $H(t)$ and $V(t)$ are real and separable.
	They are Hilbert spaces as closed subspaces of Hilbert space.
	Further it is seen that $\phi_0$ is the identity since $\Phi_0$ is defined to be the identity map, then the derivative of the identity map is the identity matrix.
	The bounds are shown in Lemmas \ref{lem:phiBoundedWithInverse} and \ref{lem:BoundednessForMapOnDual}.

	In order to show continuity of the map $t \mapsto (\|\phi_t u\|_{H(t)}, \|\phi_t v\|_{V(t)})$ for any $u \in \HO$ and $v \in \VO$, inspection of the formulae in the proof of Lemma \ref{lem:phiBoundedWithInverse} suffices, noting that all of the terms are assumed to be continuous in $t$.
	To show continuity of the dual norm, we make use of the following explicit definition of the dual norm:
	\begin{equation}
		\|\phi_t v\|_{H^{-1}} = \sup_{\xi \in H^1(\Omega_0;\R^d)} \frac{ \langle \phi_t v, \phi_t \xi\rangle}{\| \phi_t \xi \|_{H^1}}.
	\end{equation}
	Let us choose $\xi_t \in H^1(\Omega_0;\R^d)$ such that $\xi_t \neq 1$ and $\|\phi_t v\|_{H^{-1}} = \|\phi_t \xi_t\|_{H^1}^{-1}\langle \phi_t v, \phi_t \xi_t\rangle$.
	It then holds that $\|\phi_s v\|_{H^{-1}} \geq \|\phi_s \xi_t\|_{H^1}^{-1 }\langle \phi_s v, \phi_s \xi_t\rangle$.
	Furthermore,
	\begin{equation}\label{eq:dualNormContinuityCalc}
		\|\phi_t v\|_{H^{-1}} - \|\phi_s v\|_{H^{-1}}
		\leq
		\langle v, \phi_t^* \phi_t \xi_t - \phi_s^* \phi_s \xi_t \rangle \|\phi_t \xi_t\|_{H^1}^{-1} + \langle v, \phi_s^* \phi_s \xi_t \rangle \left( \frac{1}{\|\phi_t\xi_t\|_{H^1}} - \frac{1}{\|\phi_s \xi_t \|_{H^1}}\right) .
	\end{equation}
	From the regularity of $\Phi$, the form of $\phi_t^* \phi_t$, and the previously stated continuity of $t \mapsto \|\phi_t \xi \|_{H^1}$, it follows that the right hand side of \eqref{eq:dualNormContinuityCalc} tends to zero as $s \to t$.
	By repeating the same argument exchanging $s$ and $t$, it holds that $t \mapsto \|\phi_t v\|_{H^{-1}}$ is continuous.
\end{proof}.

\subsection{Time-dependent Bochner spaces}

\label{sec:Time-dependentBocnherSpaces} 
We now define the time-dependent Bochner spaces which will be of use for well-posedness.
The following definition may be found in \cite{AlpEllSti15A}.
\begin{definition}
Let $X(t)$ be a family of Hilbert spaces and $\phi_t$ a family of maps which has extension onto $X(t)^*$.
Furthermore, let
\begin{align}
L^2_X : = \{ u : [0,T] \to \cup_t X(t) \times \{ t\} , t \mapsto (\overline{u}, t) | \phi_{-(\cdot)} \overline{u} (\cdot) \in L^2(0,T;X_0)\},
\\
L^2_{X^*} : = \{ f : [0,T] \to \cup_t X^*(t) \times \{ t\} , t \mapsto (\overline{f}, t) | \phi_{-(\cdot)} \overline{f} (\cdot) \in L^2(0,T;X^*_0)\},
\end{align}
with inner products
\begin{align}
\label{eq:innerProductL2X}
(u,v)_{L^2_X} := \int_0^T (u(t), v(t))_{X(t)} \dee t
\\
\label{eq:innerProductL2XStar}
(f, g)_{L^2_{X^*}}:= \int_0^T (f(t), g(t))_{X^*(t)} \dee t.
\end{align}
\end{definition}
In the above definition, at each $t\in (0,T)$, $u(t) = (\bar{u}(t), t)$ where $\bar{u}(t)$ is an element of $X(t)$ for almost every $t$.
We will identify $u(t)$ with $\bar{u}(t)$ for convenience.
We note that our definition of $L^2_{X^*}$ does not use the map $\phi_t^*$ as appears in the corresponding definition of \cite{AlpEllSti15A}.
We are instead using the map $\phi_t$ as in Definition \ref{def:DualSpaceMapDefintion}.

From proposition \ref{prop:wholeSpaceCompatible} it follows that $L^2_{H^1}$, $L^2_{L^2}$, and $L^2_{H^-1}$ are Hilbert spaces with the associated inner products, and that $L^2_{H^1}\subset L^2_{L^2} \subset L^2_{H^{-1}}$ is a Gelfand triple.
Similarly, from \ref{prop:compatibility}, the same conclusion follows for $L^2_V \subset L^2_H \subset L^2_{V^*}$.
We note that $\left(L^2_V\right)^*$ is identified with $L^2_{V^*}$, for further properties of these spaces, we refer the reader to \cite{AlpEllSti15A}.

\subsection{Strong and weak material derivative}

\label{MaterialDerivative}
Since the domain is changing in time, it is necessary to consider the so-called material derivative, which takes into account not only the change in time of the function but the change in time of the domain.
Recall first the definition of the strong material derivative as it appears in \cite[Definition 2.20]{AlpEllSti15A}.
\begin{definition}
\label{def:StrongMaterialDerivative}
Let $(X,\phi)$ be a compatible pair, for $\xi \in C^0_X:=C^0(0,T;X)$,  we say $\xi \in C^1_X:=C^1(0,T;X)$ if the strong material derivative
\begin{equation}
\mdp{\xi} (t) := \phi_t \left( \frac{d}{dt} (\phi_{-t} \xi(t)) \right),
\end{equation}
exists with $\mdp{\xi} \in C^0_X$.
\end{definition}
We choose to use the notation $\mdp{(\cdot)}$ to emphasise the contribution of $\phi_t$ to the material derivative and further differentiate it from the standard material derivative.

Now we compute the strong material derivative induced by the Piola transformation given by $\phi_t$.
Note that this is different than the standard material derivative $\partial_\w^\bullet$ induced by the plain pull-back transformation.
For convenience, let us observe, by the definition of $\mdp{(\cdot)}$, and the product rule:
\begin{equation}\label{eq:DifferenceBetweenMaterialDerivatives}
	\begin{split}
		\mdp{\xi} &= \left(\frac{\dee}{\dee t} \left( \xi \circ \Phi_t\right) \right)\circ \Phi_{-t} + \left( D\Phi_t \frac{\dee}{\dee t}\left( D\Phi_t^{-1} \right) \right)\circ \Phi_t^{-1} \xi
		\\
		&= \partial_\w^\bullet \xi + \left( D\Phi_t \frac{\dee}{\dee t}\left( D\Phi_t^{-1} \right) \right)\circ \Phi_t^{-1} \xi.
	\end{split}
\end{equation}

It is a standard case that, if we seek a solution to Problem \ref{prob:NonZeroDirichlet} in $C^1_V$, there might not be a solution.
Essentially, it might be too much to ask for existence of a strong material derivative.
For this reason we define the weak derivative as done in \cite[Section 2.4]{AlpEllSti15A}.
The main idea of the weak derivative is to use the transport theorem \cite[Theorem 5.1]{DziEll13}, which describes how the inner product on the pivot space varies in time.
One has the following version of the Transport theorem.
\begin{theorem}%
  \label{thm:transportFormula}
	Let $u,\,v \in C_{H^1}^1$, then
	\begin{equation}\label{eq:lambdaDecomposition}
		\frac{\dee}{\dee t} \left( u(t),v(t) \right)_{L^2}
		=
		\left( \mdp{u},v \right)_{L^2} + \left(u,\mdp{v} \right)_{L^2} + \lambda(t;u(t),v(t)),
	\end{equation}
	where $\lambda(t,;\cdot,\cdot)\colon L^2(\Omega(t);\R^d) \times L^2(\Omega(t);\R^d)  \to \R$ is a bounded bilinear form given by
\begin{equation}\label{eq:definitionOfALambda}
	\lambda(t;u(t),v(t)) := \int_{\Omega(t)} u(t) \cdot \left( D\Phi_t^{-T} \frac{\dee}{\dee t} \left( D\Phi_t^T D\Phi_t \right)  D\Phi_t^{-1}  \right)\circ \Phi_t^{-1} v(t).
\end{equation}
\end{theorem}
\begin{proof}
Using that $\det(D\Phi_t) = 1$, it holds that
\begin{align}
	\frac{\dee}{\dee t} \int_{\Omega(t)} u(t)\cdot v(t)
	=&
	\frac{\dee}{\dee t}\int_{\Omega_0} (u\circ \Phi_t )\cdot (v \circ \Phi_t) \det(D\Phi_t)
	=
	\frac{\dee}{\dee t}\int_{\Omega_0} \left(D\Phi_t \phi_{-t} u\right) \cdot \left(D\Phi_t \phi_{-t} v\right)
	\\
	=&
	\int_{\Omega_0} D\Phi_t\frac{\dee}{\dee t} \left( \phi_{-t} u \right) \cdot \left( D\Phi_t \phi_{-t} v\right) + \left( D\Phi_t \phi_{-t} u\right)\cdot  D\Phi_t\frac{\dee}{\dee t} \left( \phi_{-t} v \right)
	\\
	&+ \phi_{-t} u \cdot \frac{\dee}{\dee t}\left( D\Phi_t^T D\Phi_t  \right) \phi_{-t} v \nonumber
	\\
	=&
	\int_{\Omega(t)} \mdp{u} \cdot v + u \cdot \mdp{v} + u \cdot \left( D\Phi_t^{-T} \frac{\dee}{\dee t} \left( D\Phi_t^T D\Phi_t \right)  D\Phi_t^{-1}  \right)\circ \Phi_t^{-1} v.
\end{align}
Boundedness of $\lambda$ follows from inspection and the assumed regularity of $\Phi$.
\end{proof}

Before we give the definition for a weak derivative, we define the test functions.
\begin{definition}
	We define the space $D(0,T)$ to be given by
	\begin{equation}
		D(0,T):= \{ v\colon [0,T] \to \bigcup_{t\in[0,T]}H_0^1(\Omega(t);\R^d) \times \{t\}, t \mapsto (\bar{v},t) : \phi_{-(\cdot)}\bar{v}(\cdot) \in C^1_0(0,T;H^1_0(\Omega_0;\R^d))  \}.
	\end{equation}
\end{definition}
The following definition of the weak material derivative is similar to that which appears in \cite[Definition 2.28]{AlpEllSti15A}.
\begin{definition}%
\label{def:WeakMaterialDerivative}
For $u \in
L^2_{H^1}$, the function $g \in L^2_{H^{-1}}$ is called the weak material derivative of $u$ if it holds
\begin{equation}\label{eq:weakDerivativeDef}
\int_0^T \left< g(t), \eta(t) \right>_{H^{-1}, H^1} = - \int_0^T (u(t), \mdp{\eta} (t))_{L^2} - \int_0^T \lambda(t; u(t), \eta(t)),
\end{equation}
for all $\eta \in D(0,T)$.
We write $\partial^\bullet_\phi u = g.$
\end{definition}
Let us note that the operators $\wmdw$ and $\wmdp$ are defined independently, with $\phi$ and $\w$ being different objects.
We are also abusing notation that $\wmdw$ is used for both a strong and weak material derivative.

In the definition we have given, when one takes the restriction $u \in L^2_V$, it is not immediate that $\wmdp u$ is in $L^2_{V^*}$.
For strongly differentiable functions, one may consider the following formal calculation:
\begin{equation}\begin{split}
	\Div \mdp{u}
	&=
	\Div \left( \phi_t \frac{\dee}{\dee t }\left( \phi_{-t} u\right) \right)
	=
	\left( \Div \left( \frac{\dee}{\dee t}\left( \phi_{-t} u \right)\right) \right)\circ \Phi_t^{-1}
	\\
	&=
	\left( \frac{\dee}{\dee t}\left( \Div \left( \phi_{-t} u \right)\right) \right)\circ \Phi_t^{-1}
	=
	\left( \frac{\dee}{\dee t}\left(  \left( \Div u \right)\circ \Phi_t\right) \right)\circ \Phi_t^{-1}
	\\
	&= 0,
\end{split}\end{equation}
which makes use of Lemma \ref{lem:divPreserving}, that $\Div u = 0$ and using regularity to commute $\Div$ and $\frac{\dee}{\dee t}$.
The following lemma verifies this result for weakly differentiable functions .

\begin{lemma}
	Let $\wmdp u $ be the weak derivative of $u \in L^2_V$, then it holds that $\wmdp u \in L^2_{V^*}$.
\end{lemma}
\begin{proof}
	It is immediate that $\wmdp u$ lies in $L^2_{H^{-1}}$, in order to show that $\wmdp u$ is in $L^2_{V^*}$, we are required to show that it is weakly divergence free.
	More precisely, we are required to show
	\begin{equation}
		\int_0^T \langle \wmdp u , \nabla q \rangle_{H^{-1}, H^1} = 0
	\end{equation}
	for all appropriately smooth $q$ with $\nabla q \in D(0,T)$.
	By the definition of weak derivative \eqref{eq:weakDerivativeDef} we have that
	\begin{equation}
	\int_0^T \langle \wmdp u , \nabla q \rangle_{H^{-1}, H^1}
	=
	-\int_0^T \left( u(t), \mdp{(\nabla q )}(t)\right)_{L^2(\Omega(t);\R^d)} - \int_0^T \lambda(t;u(t),\nabla q(t)).
	\end{equation}
	It is convenient to calculate $\lambda(t;u(t),\nabla q(t))$,
	\begin{equation}\label{eq:calculateLambdaForDivFree}
	\begin{split}
		\lambda(t;u(t),\nabla q(t))
		=&
		\int_{\Omega(t)} u \cdot \left( D\Phi_t^{-T} \frac{\dee}{\dee t}\left( D\Phi_t^T D \Phi_t \right) D\Phi_t^{-1} \right) \circ \Phi_t^{-1} \nabla q
		\\
		=&
		\int_{\Omega(t)} u \cdot \left(  D\Phi_t^{-T} \frac{\dee}{\dee t}\left( D\Phi_t^T\right) +  \frac{\dee}{\dee t}\left( D \Phi_t \right) D\Phi_t^{-1}  \right) \circ \Phi_t^{-1} \nabla q.
	\end{split}
	\end{equation}
	We also calculate $\mdp{(\nabla q)}$ to compare,
	\begin{equation}\label{eq:mdpNablaQ}\begin{split}
		\mdp{(\nabla q)}
		&=
		\left( D\Phi_t \frac{\dee}{\dee t}\left( D\Phi_t^{-1} \left( \left(\nabla q\right) \circ \Phi_t \right) \right) \right) \circ \Phi_t^{-1}
		\\
		&=
		\left( \frac{\dee}{\dee t} \left( \left( \nabla q \right) \circ \Phi_t \right) \right)\circ \Phi_t^{-1} - \left( \frac{\dee}{\dee t} \left( D\Phi_t \right)D\Phi_t^{-1}\right) \circ \Phi_t^{-1} \nabla q,
	\end{split}\end{equation}
	we see that the contribution of the second term of the above will cancel out the contribution of the second term of the final equality in \eqref{eq:calculateLambdaForDivFree}.
	We now calculate, for the remaining contribution of \eqref{eq:mdpNablaQ}
	\begin{equation}
	\begin{split}
		\frac{\dee}{\dee t} \left( \left(\nabla q\right) \circ \Phi_t \right)\circ \Phi_t^{-1}%
		=&
		\left(\frac{\dee}{\dee t }\left( D\Phi_t^{-T} \nabla \left(q \circ \Phi_t \right) \right) \right)\circ \Phi_t^{-1}%
		\\
		=&
		\left( \frac{\dee}{\dee t}\left( D\Phi_t^{-T} \right) \nabla \left(q \circ \Phi_t \right) + D\Phi_t^{-T} \frac{\dee}{\dee t} \nabla\left(q \circ \Phi_t \right) \right) \circ \Phi_t^{-1}
		\\
		=&
		\left( \frac{\dee}{\dee t}\left( D\Phi_t^{-T} \right) D\Phi_t^T \nabla q + D\Phi_t^{-T} \frac{\dee}{\dee t} \nabla\left(q \circ \Phi_t \right) \right) \circ \Phi_t^{-1}.
	\end{split}
	\end{equation}
	We see the contribution of the first term in the final equality of the above will cancel out with the first term of \eqref{eq:calculateLambdaForDivFree}.
	We are now left to handle the remaining term,
	\begin{equation}
		\begin{split}
			\left( D\Phi_t^{-T} \frac{\dee}{\dee t} \nabla \left(q \circ \Phi_t \right) \right) \circ \Phi_t^{-1}
			=&
			\left( D\Phi_t^{-T}  \nabla \frac{\dee}{\dee t}\left(q \circ \Phi_t \right) \right) \circ \Phi_t^{-1}
			\\
			=&
			\nabla \left( \left(\frac{\dee}{\dee t}\left(q \circ \Phi_t\right)\right)\circ \Phi_t^{-1} \right),
		\end{split}
	\end{equation}
	where we are able to exchange the order of $\frac{\dee}{\dee t}$ and $\nabla$ by the smoothness of $q \circ \Phi_t$ and we have made use of the fact that $\nabla\left(f \circ \Phi_t^{-1}\right) = \left(D\Phi_t^{-T} \nabla f \right)\circ \Phi_t^{-1}$.
	We have therefore shown that
	\begin{equation}\label{eq:FinalStepOfShowingWeakDerivativeDiv0}
		\int_0^T \langle \wmdp u , \nabla q \rangle_{H^{-1}, H^1}
		=
		- \int_0^T \left( u(t), \nabla \left(  \left(\frac{\dee}{\dee t}\left(q \circ \Phi_t\right) \right)\circ \Phi_t^{-1}\right) \right)_{L^2},
	\end{equation}
	which vanishes, since $u(t)$ is divergence free for a.e. $t \in (0,T)$ in the sense of \eqref{eq:WeaklyDivergenceFree}.
\end{proof}

\subsection{Solution space}
\label{SolutionSpace}
Having the concept of Gelfand triple of evolving Hilbert spaces and weak material derivative, we can now define the solution space, following the general concept presented in \cite[Section 2.5]{AlpEllSti15A}.

\begin{definition}[Solution space]\label{def:SolSpace}
The solution space is defined by
\begin{equation}
W(V,V^*) := \{ u \in L^2_V : \partial_\phi^\bullet u \in L^2_{V^*}\}
\end{equation}
and it is endowed with the inner product
\begin{equation}\label{eq:MovingSpacesInnerProduct}
(u, v)_{W(V,V^*) }:= \int_0^T (u(t), v(t))_{H^1} + \int_0^T ( \partial_\phi^\bullet u(t),  \partial_\phi^\bullet v(t))_{H^{-1}}.
\end{equation}
\end{definition}
In order to prove properties of the solution space, we will connect it with the standard Sobolev-Bochner space on the fixed domain, which is defined by
\begin{equation}
\mathcal{W}(\VO, \VO^*) := \{ v \in L^2(0,T; \VO ) : v' \in L^2(0,T;\VO^*)\},
\end{equation}
where the weak derivative $v'$ of $v$ is defined by
\begin{equation}\label{eq:WeakDerivativeStationarDef}
	\int_0^T \langle v',\eta\rangle_{H^{-1}, H^1} = - \int_0^T \left( u,\frac{\dee}{\dee t}\eta\right)_{L^2} \quad \forall \eta \in C^1(0,T;H^1_0(\Omega_0;\R^d)).
\end{equation}

\begin{proposition}\label{prop:evolvingSpaceEquivalence}
There is an evolving space equivalence between $W(V,V^*)$ and $\mathcal{W}(\VO, \VO^*)$ in the sense:
\begin{equation}\label{spEquivalence1}
v \in W(V,V^*) \quad \text{ if and only if } \quad \phi_{-(\cdot)}v(\cdot) \in \mathcal{W}(\VO, \VO^*)
\end{equation}
and there are $C_1,\, C_2>0$ such that
\begin{equation}\label{normEquivalence}
C_1 \| \phi_{-(\cdot)} v(\cdot) \|_{\mathcal{W}(\VO, \VO^*)} \leq \| v\|_{W(V,V^*)} \| \leq C_2 \| \phi_{-(\cdot)} v(\cdot) \|_{\mathcal{W}(\VO, \VO^*)}
\end{equation}
for all $v \in W(V,V^*)$.
\end{proposition}
\begin{proof}
Let $\tilde{u} \in W(V,V^*)$, we wish to show $\phi_{-(\cdot)} \tilde{u}(\cdot) \in \mathcal{W}(\VO,\VO^*)$.
	By Proposition \ref{prop:compatibility} and definition of $L^2_V$, it follows $\phi_{-(\cdot)}\tilde{u}(\cdot)\in L^2(0,T;\VO)$.
	We want to show that $\left(\phi_{-t}\tilde{u}(t)\right)'$ exists as a weak derivative in $L^2(0,T;\VO^*)$ in the sense of \eqref{eq:WeakDerivativeStationarDef}.

	For test function $\eta\in C^1(0,T;H_0^1(\Omega_0;\R^d))$, we wish to calculate
	\begin{equation}
		\int_0^T \int_{\Omega_0} \phi_{-t}\tilde{u} \cdot \frac{\dee}{\dee t}\eta
	\end{equation}
	and show it is appropriately bounded.
	We have, from the weak differentiability of $\tilde{u}$, that
	\begin{equation}\label{eq:EvolvingSpaceEquivalence1KnownWeakDerivative}
		\int_0^T \langle \wmdp \tilde{u}, \phi_{-t}^*\eta \rangle_{H^{-1}, H^1}
		=
		-
		\int_0^T \int_{\Omega(t)} \tilde{u} \cdot \mdplong{\left(\phi_{-t}^* \eta\right)}
		-
		\int_0^T \lambda(t;\tilde{u},\phi_{-t}^* \eta).
	\end{equation}
	From the definition of $\phi_{-t}^*$ and the product rule on the derivative $\frac{\dee}{\dee t}$, we have that
	\begin{equation}\label{eq:EvolvingSpaceEquivalence1DerivativeOfDual}
		\mdplong{\left(\phi_{-t}^* \eta\right)}
		=
		D\Phi_t^{-T} \circ \Phi_t^{-1} \left( \frac{\dee}{\dee t} \eta \right) \circ \Phi_t^{-1} + \mdplong{\left( D\Phi_t^{-T}\circ \Phi_t^{-1}\right) }\eta \circ \Phi_t^{-1},
	\end{equation}
	where we calculate
	\begin{equation}\label{eq:EvolvingSpaceEquivalence1DerivativeOfMatrix}
		\mdplong{\left( D\Phi_t^{-T}\circ \Phi_t^{-1}\right) }\eta\circ\Phi_t^{-1}
		=
		\left(D\Phi_t \frac{\dee}{\dee t}\left( D\Phi_t^{-1} D\Phi_t^{-T}\right)D\Phi_t^T  \right) \circ \Phi_t^{-1} \phi_{-t}^* \eta.
	\end{equation}
	Furthermore, it is possible to calculate
	\begin{equation}
		D\Phi_t \frac{\dee}{\dee t}\left( D\Phi_t^{-1} \right)
		=
		- \frac{\dee}{\dee t}\left( D\Phi_t\right) D\Phi_t^{-1},
	\end{equation}
	from this, one may see that
	\begin{equation}\label{eq:EvolvingSpaceEquivalence1RearraningMatrix}
	D\Phi_t \frac{\dee}{\dee t}\left( D\Phi_t^{-1} D\Phi_t^{-T}\right)D\Phi_t^T
	=
	-
	D\Phi_t^{-T} \frac{\dee}{\dee t}\left( D\Phi_t^{T} D\Phi_t \right)D\Phi_t^{-1},
	\end{equation}
	which we notice is negative the integrand of $\lambda(\cdot;\cdot,\cdot)$.
	Making use of this in \eqref{eq:EvolvingSpaceEquivalence1KnownWeakDerivative}, one has that
	\begin{equation}\begin{split}
		\int_0^T \langle \wmdp \tilde{u}, \phi_{-t}^*\eta \rangle_{H^{-1}, H^1}
		=&
		-\int_0^T \int_{\Omega(t)} \tilde{u}\cdot D\Phi_t^{-T} \circ \Phi_t^{-1} \left( \frac{\dee}{\dee t} \eta\right) \circ \Phi_t^{-1}
		\\
		=&
		-\int_0^T \int_{\Omega_0} \phi_{-t} \tilde{u} \cdot \frac{\dee}{\dee t}\eta,
		\end{split}
	\end{equation}
	which shows, by recalling the definition of weak derivative on a stationary domain \eqref{eq:WeakDerivativeStationarDef}, that $\phi_{-t} \tilde{u}$ has a weak derivative.
	In particular, for a.e. $t$, the weak derivative exists as an element of $\VO^*$ and is given by
	\begin{equation}\label{eq:EvolvingSpaceEquivalence1FinalEquality}
		\left( \phi_{-t} \tilde{u} \right)' = \phi_{-t} \wmdp\tilde{u}.
	\end{equation}
	This has shown that $\left(\phi_{-t}\tilde{u}\right)'$ exists in $L^2(0,t;\VO^*)$, and the formula \eqref{eq:EvolvingSpaceEquivalence1FinalEquality} demonstrates that there is $C>0$ such that
	\begin{equation}
		\|\phi_{-(\cdot)}\tilde{u}(\cdot)\|_{\mathcal{W}(\VO,\VO^*)} \leq C \|\tilde{u}\|_{W(V,V^*)}.
	\end{equation}

	Now we let $u \in \mathcal{W}(\VO,\VO^*)$ and want to show $\phi_{(\cdot)} u(\cdot) \in W(V,V^*)$.
	By Proposition \ref{prop:compatibility}, it follows that $\phi_{(\cdot)} u(\cdot) \in L^2_V$.
	Now, our goal is to show that $\wmdp{\left( \phi_t u\right)}$ exists as a weak derivative in $L^2_{V^*}$.

	For test functions $\eta \in D(0,T)$, we have that
	\begin{equation}\label{eq:evolvingSpaceEquivalence2FirstStep}
		\begin{split}
		\int_0^T \left( \phi_t u, \mdp{\eta} \right)_{H(t)}
		&=
		\int_0^T \int_{\Omega(t)} \left(D\Phi_t u \right)\circ \Phi_t^{-1}\cdot \left(D\Phi_t\frac{\dee}{\dee t}\left( D\Phi_t^{-1} \eta \circ \Phi_t \right) \right)\circ \Phi_t^{-1}
		\\
		&=
		\int_0^T \int_{\Omega_0} u\cdot \left(D\Phi_t^T D\Phi_t\frac{\dee}{\dee t}\left( D\Phi_t^{-1} \eta \circ \Phi_t \right) \right).
		\end{split}
	\end{equation}
	As in the previous part of this proof, we wish to transform the above so that it is $u$ multiplied against the derivative of something times $\eta \circ \Phi_t$.
	This is done in order to utilise that $u$ has a weak derivative on the stationary domain.
	In light of this, it is convenient to note that
	\begin{equation}\label{eq:evolvingSpaceEquivalence2DerivativeEta}
		\frac{\dee}{\dee t} \left( D\Phi_t^T \eta \circ \Phi_t \right)
		=
		D\Phi_t^T D\Phi_t \frac{\dee}{\dee t}\left(D\Phi_t^{-1} \eta \circ \Phi_t \right)
		+
		\frac{\dee}{\dee t}\left( D\Phi_t^T D\Phi_t \right)D\Phi_t^{-1} \eta \circ \Phi_t,
	\end{equation}
	which is a consequence of the product rule.
	Combining \eqref{eq:evolvingSpaceEquivalence2FirstStep} and \eqref{eq:evolvingSpaceEquivalence2DerivativeEta}, one has that
	\begin{equation}\label{eq:evolvingSpaceEquivalence2SecondStep}
		\int_0^T \left( \phi_t u , \mdp\eta \right)_{H(t)}
		=
		\int_0^T \int_{\Omega_0} u \cdot \left( \frac{\dee}{\dee t}\left( D\Phi_t^T \eta \circ \Phi_t \right) - \frac{\dee}{\dee t}\left( D\Phi_t^T D\Phi_t \right) D\Phi_t^{-1}\eta \circ \Phi_t\right).
	\end{equation}
	Applying the definition of the weak derivative on a stationary domain \eqref{eq:WeakDerivativeStationarDef} with test function $D\Phi_t^T \eta \circ \Phi_t$, one has
	\begin{equation}\begin{split}
		\int_0^T \left( \phi_t u , \mdp\eta \right)_{H(t)}
		=
		-\int_0^T  \langle u' , \left( D\Phi_t^T \eta \circ \Phi_t \right) \rangle_{H^{-1}, H^1} - \int_0^T\int_{\Omega_0} u \cdot \frac{\dee}{\dee t}\left( D\Phi_t^T D\Phi_t \right) D\Phi_t^{-1}\eta \circ \Phi_t
		\\
		=-\int_0^T  \langle u' , \left( D\Phi_t^T \eta \circ \Phi_t \right) \rangle_{H^{-1}, H^1} - \int_0^T \lambda(t;\phi_t u , \eta)
		\end{split}
	\end{equation}
	where we have used that $\phi_{t}^*\eta := D\Phi_t^T\eta \circ \Phi_t$.
	Therefore, by recalling the definition of a weak derivative \eqref{eq:weakDerivativeDef}, one may conclude that for a.e $t$, as an element of $V^*(t)$,
	\begin{equation}\label{eq:evolvingSpaceEquivalence2FinalEquality}
		\wmdp\left( \phi_t u\right)
		= \phi_t \left(u'\right).
	\end{equation}

	This has shown that $\wmdp \left(\phi_{t}{u}\right)$ exists in $L^2_{V^*}$, and the formula \eqref{eq:evolvingSpaceEquivalence2FinalEquality} demonstrates that there is $C>0$ such that
	\begin{equation}
		 \|\phi_{(\cdot)} u(\cdot) \|_{W(V,V^*)} \leq C\|u\|_{\mathcal{W}(\VO,\VO^*)}.
	\end{equation}

\end{proof}
\begin{remark}
We note that the above proof is very different to the result which appears in the abstract work of \cite{AlpEllSti15A}, however the result still shows a moving space equivalence.
In the abstract setting of \cite{AlpEllSti15A}, the transformations satisfy the condition that $T_t := \phi_t^{*}\phi_t$ takes $\HO$ to $\HO$.
This is not the case with the definition of $\phi_t^*$ which appears in this work.
The transformations considered by \cite{AlpEllSti15B} which are applications of the theory of \cite{AlpEllSti15A} are '\emph{orthogonal}' in the sense that $\phi_t^* = \phi_{-t}$.
In the present setting, one could potentially change the definition of $\phi_t^*$ to achieve $\phi_t^*\phi_t\colon \HO \to \HO$, however one may then lose the ability to write down a meaningful interpretation of $\phi_t^*$.
\end{remark}

An immediate corollary of Proposition \ref{prop:evolvingSpaceEquivalence} is the following.
\begin{corollary}
	The solution space $W(V,V^*)$ with inner product \eqref{eq:MovingSpacesInnerProduct} is a Hilbert space.
\end{corollary}
Furthermore, by application of \cite[Lemma 2.35]{AlpEllSti15A} we have:
\begin{lemma}\label{lem:embedsIntoContinuousSpace}
	The embedding $W(V,V^*)\subset C_H^0$ holds, that is, for any $t \in (0,T)$ and $u \in W(V,V^*)$ the map $t \mapsto u(t)$ is well defined.
\end{lemma}
This lemma also allows us to define the linear subspace
\begin{equation}\label{eq:HomogeneousMovingSpace}
	W_0(V,V^*) := \{u \in W(V,V^*) : u(0) = 0\},
\end{equation}
which is also a Hilbert space.

Analagously to this section, one may define the space $W(H^1,H^{-1})$ and prove the same results.
Due to the analytical simplicity of using a reduced velocity formulation, this space will not make an apperance until Section \ref{sec:saddle} where we will consider the saddle point formulation.

\section{Well-posedness result}\label{sec:WellPosedness}

Now that the variational evolving space framework has been set up,
following the general setting presented in
\cite[Section 4]{AlpEllSti15A},
we are able to prove well-posedness by a straight forward
application of the
Banach-Ne\v cas-Babu\v ska Theorem.
The operator formulation of the problem we consider may be given as
\begin{equation}\label{eq:AmalProblem}\tag{\textbf{P}}\begin{split}
L \wmdp{u} + Au + \Lambda u &= f  \mbox{ in } L^2_{V^*},
\\
u(0) &= u_0 \in \HO,
\end{split}
\end{equation}
where $L\colon L^2_{H^{-1}} \to L^2_{H^{-1}}$ and $A,\,\Lambda\colon L^2_{H^1} \to L^2_{H^{-1}}$.
We compare \eqref{eq:AmalProblem} and Problem \ref{prob:ZeroDirichlet} to see that for each $t \in (0,T)$ and any $v \in W(V,V^*)$,
\begin{align}
	(\Lambda v)(t):=& \left( D\Phi_t^{-T} \frac{\dee}{\dee t}\left( D\Phi_t^T D \Phi_t \right)D\Phi_t^{-1}\right)\circ \Phi_t^{-1} v (t),\label{eq:LambdaDef}
	\\
	(Av )(t):=& -\Delta v(t) + \left( \left(\Vee (t)-\w(t)\right) \cdot \nabla \right) v(t) - \left( D\Phi_t \frac{\dee}{\dee t}\left( D\Phi_t^{-1} \right)  \right)\circ \Phi_t^{-1} v - (\Lambda v)(t),\label{eq:ADef}
	\\
	L :=& \Id_{L^2_{H^{-1}}}.\label{eq:LDef}
\end{align}
In the above, we have made use of the characterisation of $\mdp{(\cdot)}$ noted in \eqref{eq:DifferenceBetweenMaterialDerivatives} which may be transfered to the weak setting.
Notice that $\Lambda$ has been removed from $A$, this is so that we may keep our notation and calculations as similar to \cite{AlpEll15} as possible.
With these operators in mind, we give the definition of the following bilinear forms.

\begin{definition}\label{def:weakBilinearForms}
	For each $t \in (0,T)$ we define
  the duality pairing
	\begin{align}
		l(t;\cdot,\cdot)
    := \langle \cdot, \cdot \rangle_{H^{-1}, H^1} &\colon H^{-1}(\Omega(t);\R^d) \times H^1_0(\Omega(t);\R^d)\to \R,
	\end{align}
  and the bilinear forms
	\begin{align}
		b(t;\cdot,\cdot) &\colon H^1(\Omega(t);\R^d) \times H^1(\Omega(t);\R^d) \to \R,
		\\
		a(t;\cdot,\cdot) &\colon H^1(\Omega(t);\R^d) \times H^1(\Omega(t);\R^d) \to \R,
	\end{align}
	by
	\begin{align}
		b(t;u,v)&:= \left(u,v\right)_{H^1} - \left(u,v\right)_{L^2}
		+\left(\left(\left(\Vee -\w\right)\cdot \nabla \right) u, v\right)_{L^2}
    &&\forall u,\,v\in H^1(\Omega(t);\R^d)\nonumber
		\\
		a(t;u,v)&:=
		b(t;u,v)
		- \left( \left(D\Phi_t\frac{\dee}{\dee t}\left(D\Phi_t^{-1}\right) \right)\circ \Phi_t^{-1} u, v\right)_{L^2} - \lambda(t;u,v)
    &&\forall u,\,v\in H^1(\Omega(t);\R^d)\nonumber
	\end{align}
	where we also recall the definition for $\lambda(t;\cdot,\cdot)$ as given in \eqref{eq:definitionOfALambda}.
\end{definition}
Notice that $(u,v)_{H^1} - (u,v)_{L^2} = \int_{\Omega(t)} D u : D v$.

We now require some assumptions on $\Vee$ to show appropriate properties of the bilinear forms $a$ and $b$ for the well-posedness of our system.
\begin{assumption}\label{ass:AssumptionsForV}
	Assume that $\Vee$ satisfies
	one of the following:
	\begin{itemize}
	\item
		$\Vee\circ \Phi_{(\cdot)} \in L^\infty(0,T;L^p(\Omega_0 ))$, where $p\geq d$ and $\Div \Vee = 0$ weakly;
	\item
		$\Vee\circ \Phi_{(\cdot)} \in L^\infty(0,T;L^\infty(\Omega_0))$.
	\end{itemize}
\end{assumption}
These conditions are required to show the final point of the following Proposition, which gives the properties on the bilinear forms.
It is worth mentioning that the edge case $\Vee = \w$ satisfies this assumption as does, for $d = 2$, the case $\Vee = u$, where $u$ is the solution to our parabolic problem.
This case may be of interest when considering a moving domain Navier--Stokes problem.

\begin{proposition}\label{prop:propertiesOfBilinearForms}
The bilinear forms $a$, $b$, $\lambda$, and $l$ satisfy the following conditions:
	\begin{enumerate}
		\item The maps
		\begin{align}
			 t \mapsto&\,  b(t;u(t),v(t)) \quad \forall u,v \in L^2_{H^1},
			\\
			 t \mapsto&\,  a(t;u(t),v(t)) \quad \forall u,v \in L^2_{H^1},
			\\
			t \mapsto& \, \lambda(t;u(t),v(t)) \quad \forall u,v \in L^2_{L^2},
			\\
			t \mapsto& \, l(t; u(t), v(t)) \quad \forall (u,v) \in L^2_{H^{-1}}\times L^2_{H^1}
		\end{align}
		are measurable.
		\item The bilinear forms $a$, $b$, $\lambda$, and $l$ are bounded uniformly in $t$.
		\item There is $C_1,\, C_2>0$ such that for any $t\in(0,T)$, $u\in H^1(\Omega(t);\R^d)$,
		\begin{equation}\label{eq:WeakerCoercivityCondition}
			b(t,u,u) \geq C_1 \|u\|_{H^1}^2 - C_2 \|u\|_{L^2}^2.
		\end{equation}
		\item There is $C_1,\, C_2>0$ such that for any $t\in(0,T)$, $u\in H^1(\Omega(t);\R^d)$,
		\begin{equation}\label{eq:WeakerCoercivityCondition2}
			a(t,u,u) \geq C_1 \|u\|_{H^1}^2 - C_2 \|u\|_{L^2}^2.
		\end{equation}
	\end{enumerate}
\end{proposition}

\begin{proof}

	{\it Part 1:}
    Since $t \mapsto u(t), v(t)$ are measurable functions, the measurability of the  bilinear forms $b(t;u(t),v(t)$ and $l(t; u(t), v(t)) $ follows directly from their definitions and properties of measurable functions. Furthermore, the measurability of t $\mapsto \, \lambda(t;u(t),v(t))$, can be proved in an analogue way as presented in \cite[Lemma 2.26]{AlpEllSti15A}, which directly implies the measurability of the $  a(t;u(t),v(t)) $.
	\\{\it Part 2:}
    The boundedness of $l$ follows from inspection.
    For the boundedness of $\lambda$, one appeals to the fact that $\Phi_t$ is regular in time and space.
    In order to see that $b$ and $a$ are bounded, it is enough to prove boundedness of the term $\left((\Vee-\w)\cdot \nabla \right)u \cdot v$.
    For this, it is sufficient to only consider the contribution due to $\Vee$ since we are assuming $\w$ to be regular in space and time.
    The \emph{minimal} assumption in Assumption \ref{ass:AssumptionsForV} is that $\Vee\circ \Phi_{(\cdot)} \in L^\infty(0,T;L^p(\Omega_0))$ for $p\geq d$, therefore
    \[
    	\int_{\Omega(t)}\left(\Vee(t)\cdot \nabla \right)u \cdot v \leq \|\Vee(t)\|_{L^p} \|u\|_{H^1} \|v\|_{L^{p^*})},
    \]
    where $\frac{1}{p}+ \frac{1}{p^*} = \frac{1}{2}$ and $p\geq d$.
    By Sobolev embedding, we see that there is $C>0$ such that
    \[
    	\|v\|_{L^{p^*}} \leq C \|v\|_{H^1}
    \]
    which shows boundedness of $b$ and $a$.
    \\{\it Part 3:}
    For the weaker coercivity condition \eqref{eq:WeakerCoercivityCondition},
    we first note that $\Div w=0$ implies that
    $((w\cdot\nabla)u,u)_{L^2}=0$,
    by an integration by parts argument and $u|_{\partial \Omega(t)}=0$.
    Hence the difficulty is again focused on $\Vee$.
    We consider the two cases of Assumption \ref{ass:AssumptionsForV} separately.
    If $\Div \Vee =0$, the same integration by parts argument shows that
    $(((\Vee-w)\cdot\nabla)u,u)_{H(t)}=0$,
    from which the result follows trivially with $C_1=C_2=1$.
    In the case that $\Div \Vee \neq 0$ we use the bound
    \[
      \left|\int_{\Omega(t)}\left(\Vee(t)\cdot \nabla \right)u \cdot u \right|
        \leq \|\Vee(t)\|_{L^\infty} \|u\|_{H^1} \|u\|_{L^2}
        \leq  \frac{1}{2}\|u\|_{H^1}^2 + \frac{1}{2}\|\Vee(t)\|_{L^\infty}^2\|u\|_{L^2}^2
    \]
    to obtain the result with constants
    \begin{equation}
      C_1 = \frac{1}{2} \text{ and }
      C_2 = 1+ \frac{1}{2}\|\Vee(t)\|_{L^\infty}.
    \end{equation} 
  \\{\it Part 4:}
  Finally, weaker coercivity of $a$ follows from
  weaker coercivity of $b$ and the fact that
  regularity of $\Phi_t$ in space and time provides
  a bound for the difference
  \begin{equation}
    |a(t;u,u)-b(t;u,u)| \leq C \|u\|_{L^2}^2
  \end{equation}
  with some $\Phi$-dependent constant $C>0$.
  
\end{proof}
With the bilinear forms defined and their properties given, we may state the abstract problem we wish to solve.

\begin{problem}\label{prob:preciseWeakFormulationMovingDomain}
	Given $f \in L^2_{V^*}$ $u_0 \in \HO$, find $u \in W(V,V^*)$ such that $u|_{t=0} =u_0$
	\begin{equation}\label{eq:weak_problem}
		\int_0^T l(t;\wmdp{u}(t),v(t)) + a(t;u(t),v(t)) + \lambda(t;u(t),v(t)) \dee t = \langle f,v \rangle_{L^2_{H^{-1}},L^2_{H^1}} \quad \forall v \in L^2_V.
	\end{equation}
\end{problem}
\changed{
\begin{remark}\label{rem:time_dept}
	Notice that in the above formulation, other than the data $f, \Vee$, and $\w$, there is minimial time dependence.
	Assuming appropriate smoothness, one may reformulate \eqref{eq:weak_problem}
	using \eqref{eq:DifferenceBetweenMaterialDerivatives} into
	\begin{equation*}
    \int_0^T l(t;\wmdw{u}(t),v(t)) + b(t;u(t),v(t)) \dee t = \langle {f},v \rangle_{L^2_{H^{-1}},L^2_{H^1}}
    \qquad \forall v \in L^2_V,
	\end{equation*}
	which particularly emphasises the time dependence, or lack thereof.
	This form, however, is theoretically less convienient as it need not hold $\wmdw{u} \in L^2_{V^*}$ in this setting. 
	This particular form may be more useful for a mixed formulation, whereby one may only be interested in $\wmdw{u} \in L^2_{H^{-1}}$.
\end{remark}}
A standard method is to seek this $u$ such that it has decomposition:
\begin{equation}\label{eq:decomposition}
	u(t) = \tilde{u}(t) + \tilde{y}(t),
\end{equation}
where $\tilde{u} \in W_0(V,V^*)$ and $\tilde{y} = \phi_{(\cdot)} y$, for $y\in \mathcal{W}(\VO,\VO^*)$ with $y(0) = u(0)$, which may potentially be chosen as the solution of an appropriate PDE, say a parabolic Stokes equation on $\Omega_0$.
It may be seen that, after relabelling, it sufficient to seek $u \in W_0(V,V^*)$.

\begin{problem}\label{prob:preciseWeakFormulationMovingDomainInitialData0}
	Given $f \in L^2_{V^*}$ $u_0 \in \HO$, find $u \in W_0(V,V^*)$ such that
	\begin{equation}
		\int_0^T l(t;\wmdp{u}(t),v(t)) + a(t;u(t),v(t)) + \lambda(t;u(t),v(t)) \dee t = \langle \tilde{f},v \rangle_{L^2_{H^{-1}},L^2_{H^1}}\quad \forall v \in L^2_V,
	\end{equation}
	where
	\begin{equation}
		\langle \tilde{f},v \rangle_{L^2_{H^{-1}},L^2_{H^1}}
		:=
		\langle f,v \rangle_{L^2_{H^{-1}},L^2_{H^1}} - \int_0^T\left( l\left(t;\wmdp{\tilde{y}}(t),v(t)\right) + a\left(t;\tilde{y}(t),v(t)\right) + \lambda\left(t;\tilde{y}(t) ,v(t)\right)\right)\dee t.
	\end{equation}
\end{problem}
Our well-posedness, as in \cite{AlpEllSti15A}, follows from an application of the Banach-Ne\v{c}as-Babu\v{s}ka theorem.
	The abstract theorem is given as:
\begin{theorem}[Banach-Ne\v{c}as-Babu\v{s}ka Theorem]\label{thm:BNB}
Let $X$ be a Banach space and $Y$ be a reflexive Banach space.
Let
$B:X \times Y \to \R$ a bounded bilinear form
  and $F \in Y^*$.
Then there is a unique $u_F \in X$ such that
\begin{equation}
  \label{eq:BNB_problem}
	B(u_F,v) = F(v) \mbox{ for all } v \in Y
\end{equation}
if and only if
\begin{equation}\label{eq:BNBInf-Sup}
	\exists C >0 : \forall u \in X,~ \sup_{v \in Y} \frac{B(u,v)}{\|v\|_Y}\geq C \|u\|_X,
\end{equation}
\begin{equation}\label{eq:BNBInjectiveDual}
	\forall v \in Y, ~\left( \forall u \in X, ~ B(u,v) = 0\right) \implies v = 0.
\end{equation}
\end{theorem}

A proof of this may be found in \cite{ErnGue04}.
We note that Problem~\ref{prob:preciseWeakFormulationMovingDomainInitialData0} is of the form \eqref{eq:BNB_problem} for the bilinear form
$B:X\times Y \to \R$ defined by
\begin{equation}
	B(u,v):=\int_0^T l(t;\wmdp{u}(t),v(t)) + a(t;u(t),v(t)) + \lambda(t;u(t),v(t))\dee t
\end{equation}
on the spaces
$X= W_0(V,V^*)$ and $Y = L^2_V$
and the right hand side $F=\tilde{f}$.
Notice that showing well-posedness of of the problem
now amounts in summarising that Proposition~\ref{prop:propertiesOfBilinearForms}
guarantees \eqref{eq:BNBInf-Sup} and \eqref{eq:BNBInjectiveDual}
and thus
the applicability of the abstract theorem.

\begin{theorem}\label{thm:ExistsSolutionToOurProblem}
	There is a unique solution to Problem \ref{prob:preciseWeakFormulationMovingDomain}.
\end{theorem}
\begin{proof}
  We apply Theorem \ref{thm:BNB} to Problem \ref{prob:preciseWeakFormulationMovingDomainInitialData0}
  with $X$, $Y$, $B:X\times Y \to \R$, and $F$ as given above.
  To this end it is sufficient to show the inf-sup-type condition \eqref{eq:BNBInf-Sup},
  the dual injectivety condition \eqref{eq:BNBInjectiveDual},
  and $F \in Y^*$.

  Thanks to the properties of the bilinear forms shown in
  Proposition~\ref{prop:propertiesOfBilinearForms}, conditions
  \eqref{eq:BNBInf-Sup} and \eqref{eq:BNBInjectiveDual}
  follow directly from Lemmas 4.3 and 4.4 in \cite{AlpEllSti15A},
  respectively, where a slightly different notation (operators instead
  of bilinear forms) is used.
  
  Finally, one may verify that $\tilde{f}\in L^2_{V^*}$ since $\tilde{y}$, which appears in \eqref{eq:decomposition} satisfies $\tilde{y}\in W(V,V^*)$.
  This gives us the $\tilde{u} \in W_0(V,V^*)$, thus by \eqref{eq:decomposition}, we recover $u \in W(V,V^*)$ with $u(0) = u_0$.
\end{proof}
This has given the existence of a weak solution to the parabolic Oseen equation in a moving domain in a reduced velocity formulation.

\subsection{Recovering a solution to the saddle point formulation}\label{sec:saddle}
In the above analysis, we considered a reduced velocity formulation of a parabolic Oseen equation.
It is then natural to ask if there is a solution to the formulation with pressure.
For convenience, let us define the pressure space
\begin{equation}
	L^2_{L^2_0} := \{ q\colon [0,T] \to \bigcup_{t \in [0,T]} L^2(\Omega(t)) \times \{t\} : q(\cdot) \circ \Phi_{-(\cdot)} \in L^2((0,T); L^2(\Omega_0)), \int_{\Omega(\cdot)}q(\cdot) = 0\ a.e. \}.
\end{equation}
The answer to this question is given in the following corollary to Theorem \ref{thm:ExistsSolutionToOurProblem}.
\begin{corollary}\label{cor:saddle}
	Let $u_0 \in \HO$ and $f \in L^2_{V^*}$ and let $L^2_0(\Omega_0)$.
    Then there exist unique $(u,p)$ such that $u \in W(H^1,H^{-1})$, $p \in L^2_{L^2_0}$, $u|_{t=0} = u_0$, and
	\begin{align}
		\label{eq:SaddlePoint1}
		\int_0^T l(t;\wmdp{u}(t),v(t)) + a(t;u(t),v(t)) + \lambda(t;u(t),v(t)) + \int_{\Omega(t)} p(t) \Div v(t)  \dee t &= \langle f,v \rangle_{L^2_{H^{-1}},L^2_{H^1}},
		\\
		\label{eq:SaddlePoint2}
		\int_0^T \int_{\Omega(t)} q(t) \Div u(t) &= 0
	\end{align}
	for all $v \in L^2_{H^1_0}$ and for all $q$ such that $q \circ \Phi_{(\cdot)} \in L^2( 0,T; L^2_0(\Omega_0))$.
    In particular, $u$ is a solution to Problem \ref{prob:preciseWeakFormulationMovingDomain}.
\end{corollary}
The proof of the result follows almost identically to the proof of Theorem 5.1 in \cite{VouReu18}, where one must make appropriate changes for the fact our domain moves.
An outline of the proof relies on taking the solution $u \in W(V,V^*)$ to Problem \ref{prob:preciseWeakFormulationMovingDomain} and grouping up all the terms involving $u$ from \eqref{eq:SaddlePoint1} as the linear operator $\ell_u \in (L^2_{L^2_0})^*$ which satisfies
\begin{equation}\label{eq:toApplydeRham}
	\int_0^T \int_{\Omega(t)} p(t) \Div v(t) = \ell_u(v) \quad \forall v \in L^2_{H^1_0}
\end{equation}
and we note that $\ell_u|_{L^2_V} = 0$.
From this, one uses an adaptation of de Rham's theorem to invert the linear operator in \eqref{eq:toApplydeRham}.
We refer the reader to Corollary 2.4 of \cite{GroReu11} for a precise statement of de Rham's theorem which roughly states that $\nabla \colon L^2_0(\Omega(t))\to \{ g \in H^{-1}(\Omega(t)) : g|_{V(t)} =0 \}$ is an isomorphism for each $t \in (0,T)$.

\section{Discretisation in time}\label{sec:TimeDiscretisation}

In this section we wish to give an analysis of a time discrete system related to Problem \ref{prob:NonZeroDirichlet}.
The discretisation of evolving space PDEs has been studied extensively in \cite{EllRan20}, the study of a fully discrete system for a heat equation was considered in \cite{DziEll12} and for linear parabolic equations in \cite{LubManVen13} the full discretisation is also considered.

We neglect the full discretisation and focus on the time discretisation.
This choice is made as the significant difference in this article to previous articles considering evolving Bochner spaces is the construction of the time derivative.
Of course one might consider the spacial discretisation to obtain a full discretisation, however we comment that this would not prove interesting or original since standard results can be applied.
Other methods are certainly of interest, making use of unfitted meshes \cite{BurFreMas19,WahRicLeh21}.

For our system, we suggest the following time update step.
For $n \geq 0$, let $t_{n} = n \tau$ for some $\tau>0$.
Given $u^n \in V(t_n)$, find $u^{n+1} \in V(t_{n+1})$
such that
\begin{equation}\label{eq:DiscreteEquation}\begin{split}
	\int_{\Omega(t_{n+1})} u^{n+1} \cdot \eta + \tau \nabla u_{n+1} : \nabla \eta &+ \tau\left((\Vee-\w)(t_{n+1})\cdot \nabla\right) u^{n+1} \cdot \eta
	\\&=
	\int_{\Omega(t_n)} u^n \cdot \left( \eta \circ \Phi_{t_{n+1}} \circ \Phi_{t_n}^{-1}\right) + \tau \int_{\Omega(t_{n+1})}f \cdot \eta
	\end{split}
\end{equation}
for all $\eta \in V(t_{n+1})$.
\begin{remark}
	Despite the inconvenient terms involving $D\Phi_t$ and derivatives which appear in the weak problem, Problem \ref{prob:preciseWeakFormulationMovingDomain}, we see that this discretisation has the form of a '\emph{standard}' moving domain discretisation.
	That is to say the time discretisation we present appears to be the time discretisation one might propose from looking at \eqref{eq:ParabolicOseen}.
  Let us note that the discretisation we provide is first order.
  Higher order discretisations are certainly possible, these could use higher order finite difference schemes, or even Discontinuous Galerkin strategies.
  \changed{However such schemes may or may not require the inclusion of the aforementioned inconvenient terms involving $D \Phi_t$.}
\end{remark}

We now provide a justification for this proposed discretisation.
\subsection{Derivation of discretisation}
We recall the explicit form of Problem \ref{prob:preciseWeakFormulationMovingDomain} for sufficiently smooth data $f$,
\[
	\int_0^T \int_{\Omega(t)} \wmdp u \cdot \eta - \left( D\Phi_t (D\Phi_t^{-1})'\right)\circ \Phi_t^{-1} u \cdot \eta + \nabla u : \nabla \eta + \left((\Vee-\w)\cdot\nabla \right)u \cdot \eta = \int_0^T\int_{\Omega(t)} f \eta
\]
for any $\eta \in L^2_V$.
Let us assume that $u$ is sufficiently smooth so that $\wmdp u = \mdp{u}$.
Using the transport formula from Theorem~\ref{thm:transportFormula} and rewriting $\lambda$ using the chain rule yields
\[
	\frac{\dee }{\dee t}\int_{\Omega(t)} u \cdot \eta
	=
	\int_{\Omega(t)} \mdp {u} \cdot \eta + u \cdot \mdp{\eta} - \left( D\Phi_t (D\Phi_t^{-1})' \right) \circ \Phi_t^{-1} u \cdot \eta - u \cdot \left( D\Phi(D\Phi_t^{-1})' \right) \circ \Phi_t^{-1} \eta.
\]
Choosing $\eta$ to satisfy $\mdp{\eta} = 0$, one arrives at
\begin{equation}\label{eq:EquationWhichWeDiscretise}
	\frac{d}{dt}\int_{\Omega(t)} u\cdot \eta
	+
	\int_{\Omega(t)} u \cdot \left( D\Phi_t (D\Phi^{-1}_t)' \right) \circ \Phi_t^{-1} \eta + \nabla u : \nabla \eta + \left((\Vee-\w)\cdot\nabla \right) u \cdot \eta = \int_{\Omega(t)} f\cdot\eta
\end{equation}
for almost every $t \in (0,T)$.
We will now approximate the time derivatives which appear in \eqref{eq:EquationWhichWeDiscretise}.
It is standard to approximate the time derivative of the inner product as:
\begin{equation}\label{eq:ApproxDerivativeOfInnerProduct}
	\frac{d}{dt}\int_{\Omega(t)} u \cdot \eta
	\approx
	\frac{1}{\tau} \left( \int_{\Omega(t_{n+1})} u(t_{n+1}) \cdot \eta(t_{n+1}) - \int_{\Omega(t_n)} u(t_n)\cdot \eta(t_n)\right).
\end{equation}
Notice that in \eqref{eq:EquationWhichWeDiscretise} we have the time derivative of $D\Phi_t^{-1}$.
This term could be included as is, providing a different discretisation to \eqref{eq:DiscreteEquation}, however if one were to work with an unknown $\w$, hence unknown $\Phi$, it may not be convenient to directly use $(D\Phi_t^{-1})'$, so discretisation may be appropriate.
Here, we use a first order approximation $D\Phi_t (D\Phi_t^{-1})' = -(D\Phi_t)' D\Phi_t^{-1} \approx -\frac{1}{\tau} (D\Phi_{t_{n+1}} D\Phi_{t_n}^{-1} - I)$,
therefore
\begin{equation}\label{eq:ApproxHorribleTerm}
	\int_{\Omega(t)} u \cdot \left( D\Phi_t (D\Phi_t^{-1})' \right) \circ \Phi_t^{-1} \eta
	\approx
	\frac{1}{\tau}\int_{\Omega(t_n)} u(t_n) \cdot \eta(t_n) - u(t_n) \cdot (D\Phi_{t_{n+1}} D\Phi_{t_n}^{-1}) \circ \Phi_{t_n}^{-1} \eta(t_n),
\end{equation}
where the first term of this will cancel with the second term from the time derivative of the integral.

It is also convenient to calculate what $\eta(t_n)$ is in terms of $\eta(t_{n+1})$ under the relationship that $\eta(t) = \phi_t\eta_0$ for some given $\eta_0 \in V(0)$,
\begin{equation}\label{eq:whatIsEtaTn}
	\eta(t_n)
	=
	\phi_{t_n} \phi_{-t_{n+1}} \eta(t_{n+1})
	=
	\left( D\Phi_{t_n} D\Phi_{t_{n+1}}^{-1} \eta(t_{n+1}) \circ \Phi_{t_{n+1}}\right)\circ \Phi_{t_n}^{-1}.
\end{equation}
This results in
\begin{equation}
	(D\Phi_{t_{n+1}} D\Phi_{t_n}^{-1}) \circ \Phi_{t_n}^{-1} \eta(t_n)
	=
	\eta(t_{n+1}) \circ \Phi_{t_{n+1}} \circ \Phi_{t_n}^{-1}.
\end{equation}

As such, when using \eqref{eq:ApproxDerivativeOfInnerProduct}, \eqref{eq:ApproxHorribleTerm} and \eqref{eq:whatIsEtaTn} in \eqref{eq:EquationWhichWeDiscretise} and approximating the other terms implicitly, we arrive at the discretisation in \eqref{eq:DiscreteEquation}.

\subsubsection{Existence and uniqueness of discrete solution}
It is of course necessary to demonstrate that for appropriate data, there is a solution to the discrete system.
It may be seen that this is the case, as the strong form of the problem posed in \eqref{eq:DiscreteEquation}  is given by:
\begin{align}
	-\Delta u^{n+1} + ((\Vee - \w)(t_{n+1})\cdot \nabla) u^{n+1} + \frac{1}{\tau} u^{n+1} + \nabla  p^{n+1} &= \frac{u^n\circ \Phi_{t_n}\circ \Phi_{t_{n+1}}^{-1}} {\tau}+f(t_{n+1}) \quad \mbox{ in } \Omega(t_{n+1})
	\label{eq:discreteSaddle}
	\\
	\Div u^{n+1} &= 0 \quad \mbox{ in } \Omega(t_{n+1}),
	\\
	u^{n+1}|_{\partial \Omega(t_{n+1})} &= 0.
\end{align}
From this form, with suitable assumptions on $\Vee$ which will be given in Assumption \ref{ass:AssumptionForDiscrete}, it follows that there exists a unique solution $(u^{n+1},p^{n+1})$ to the above system.

\subsubsection{Interpolation of discrete solutions}
For $n\geq 0$, for $t \in (t_n,t_{n+1})$, if one is interested in the interpolation between $u^n \in V(t_n)$ and $u^{n+1} \in V(t_{n+1})$, we note that it is not necessarily appropriate to consider the standard interpolation given by $\frac{1}{t_{n+1}-t_n}\left(u^{n+1}\circ \Phi_{t_{n+1}} (t - t_{n})  +u^n \circ \Phi_{t_n} (t_{n+1} -t)  \right)\circ \Phi_{-t}
\not \in V(t)$.
Instead, one should consider $\frac{1}{t_{n+1}-t_n}\phi_{t} \left( \phi_{-t_{n+1}}u^{n+1} (t-t_n) + \phi_{-t_n} u^n (t_{n+1}-t) \right) \in V(t)$.

\subsection{Proof of convergence of discretisation}

We now prove, under certain regularity assumptions that the above discretisation will converge as $\tau \to 0$.
Due to the moving space framework, this essentially boils down to calculating the error for an ODE.

Recall that $\mdp{u} - (D\Phi_t(D\Phi_t^{-1})')\circ \Phi_t^{-1}u  = \wmdw u$.
We will use this to make the calculations somewhat shorter.
It is convenient to use the following alternate formulation of the continuous equation:
\begin{equation}\label{eq:ContinuousEquationAlternateFormPointwise}
 \int_{\Omega(t)}\wmdw u \cdot \eta + \nabla u : \nabla \eta + \left((\Vee-\w)\cdot \nabla \right)u \cdot \eta = \int_{\Omega(t)}f \cdot \eta
\end{equation}
for $\eta \in V(t)$ for a.e. $t \in (0,T)$.

\begin{assumption}\label{ass:AssumptionForDiscrete}
	We assume that $f$ is continuous in time and $L^2$ in space;
	$\Vee$ satisfies either $\Vee \circ \Phi_{(\cdot)} \in C(0,T; L^p(\Omega_0))$ with $\Div \Vee =0$ for $p\geq d$ or $\Vee \circ \Phi_{(\cdot)} \in C(0,T;L^\infty(\Omega_0))$;
	$u$, the solution to Problem \ref{prob:preciseWeakFormulationMovingDomainInitialData0} satisfies $u \in C^2_H$;
	Equation \eqref{eq:ContinuousEquationAlternateFormPointwise} holds for every $t \in (0,T)$.
\end{assumption}
We note that the assumptions made on $f$ and $\Vee$ are not expected to be sufficient to ensure the assumptions required on the solution $u$.

Furthermore we assume that the solution is sufficiently smooth to guarantee
consistency with order $r>0$ of the backward Euler method
in the sense that its pullback by $\hat{u}(t) =  u(t) \circ \Phi_t$ satisfies
\begin{equation}\label{eq:assumedError}
  \Bigl\|\hat{u}(t+\tau) - \hat{u}(t) - \tau \frac{\dee \hat{u}}{\dee t}(t+\tau)\Bigr\|_{L^2(\Omega(0))} \leq C(u) \tau^{r+1}.
\end{equation}
This can e.g.\ be shown with order $r=1$ if $\hat{u}$ is $C^2(0,T;L^\infty(\Omega(0)))$.
For later reference we note that
the time derivative of the pullback is given by the pullback of the material derivative,
that is
\begin{equation}
  \frac{\dee}{\dee t} \hat{u}(t)
  = \Bigl(\frac{\dee }{\dee t} u(t) + \nabla u(t) \cdot \w(t)\Bigr)\circ \Phi_t
  = \wmdw u(t) \circ \Phi_t.
\end{equation}
Thus the assumed consistency estimate can be written as
\begin{equation}
  \Bigl\|u(t_{n+1})
  -u(t_{n}) \circ \Phi_{t_{n}} \circ \Phi_{t_{n+1}}^{-1}
  - \tau\wmdw u(t_{n+1})\Bigr\|_{L^2(\Omega(t_{n+1}))}
  \leq C(u) \tau^{r+1},
\end{equation}
where we have again utilised that $\det(D\Phi_t) \equiv 1$.

We will also assume that $\tau$ is sufficiently small.

\begin{theorem}
  Let $u \in W(V,V^*)$ be the solution to Problem \ref{prob:preciseWeakFormulationMovingDomain}, let $\{u^n\}_{n\geq 1}$ be the solutions to the discrete system \eqref{eq:DiscreteEquation} with the same initial data.
  Under the assumption that \eqref{eq:assumedError} holds with $r\in (0,1]$ and the conditions of Assumption \ref{ass:AssumptionForDiscrete} it holds that there is $C>0$ independent of $\tau$ such that for each $n \geq 0$ with $\tau n < T$,
  \[
    \|u^{n+1}-u(t_{n+1})\|_{L^2(\Omega(t_{n+1}))}^2 + \tau \|\nabla \left( u^{n+1}-u(t_{n+1}) \right) \|_{L^2(\Omega(t_{n+1}))}^2
    \leq C \tau^{2 r}.
  \]
\end{theorem}
Before we prove this statement, let us note that this result follows exactly the same argumentation as if one were using a discretisation based on the saddle point formulation, rather than the reduced velocity formulation provided.
This is due to the lack of spatial discretisation, whereby in full discretisations, one may not have exactly divergence free velocity fields.

\begin{proof}

  We aim to estimate the quantities $\|u^{n+1} - u(t_{n+1})\|_{L^2}$ and $\|\nabla (u^{n+1} - u(t_{n+1}))\|_{L^2}$.
  For convenience, let us write $e_{n}:= u^{n} - u(t_{n})$ and $\|\cdot\|_n = \|\cdot\|_{L^2(\Omega(t_{n}) )}$.
  In the following we will make use of the fact that
  $\|v\|_n = \|v\circ \Phi_{t_n}\|_{L^2(\Omega_0)} = \|v\circ \Phi_{t_n} \circ \Phi_{t_{n+1}}^{-1}\|_{n+1}$,
  which holds because $\det(D\Phi_t) =1 $, which is shown in Lemma \ref{lem:propertiesOfDomainAndDet}.
  
  By argumentation similar to that which appears in Proposition \ref{prop:propertiesOfBilinearForms}, we have that there are constants $\alpha,\, \beta >0$ such that
\begin{equation}\label{eq:WeakCoercivityDiscretisation}
	\alpha \int_{\Omega(t_{n+1})} |\nabla\eta|^2
	\leq
	\int_{\Omega(t_{n+1})} \left( |\nabla\eta|^2
	+
	\left( (\Vee-\w)(t_{n+1})\cdot \nabla\right)\eta  \cdot  \eta	\right)
	+ \beta
	\int_{\Omega(t_{n+1})} \eta^2
\end{equation}
for all $\eta \in H^1(\Omega(t_{n+1});\R^d)$.
The estimate \eqref{eq:WeakCoercivityDiscretisation} follows from the assumptions on $\Vee$ made in Assumption \ref{ass:AssumptionForDiscrete}.
  
  In order to estimate
  $\| e_{n+1} \|_{n+1}$ and $\| \nabla e_{n+1} \|_{n+1}$,
  it is convenient to define $I$ as
  \begin{equation}\label{eq:Quantitity to estimate}
    I:=
    \| e_{n+1} \|_{n+1}^2 + \alpha \tau \| \nabla e_{n+1} \|_{n+1}^2
    =
    \int_{\Omega(t_{n+1})} e_{n+1}^2
    +
    \alpha \tau \int_{\Omega(t_{n+1})} |\nabla e_{n+1}|^2.
  \end{equation}

We use this weaker coercivity \eqref{eq:WeakCoercivityDiscretisation} with $\eta = e_{n+1}$ to estimate $I$, we also rearrange the products into a form which will be convenient
\begin{equation}\begin{split}
	I
	\leq&
	\int_{\Omega(t_{n+1})} u^{n+1} \cdot e_{n+1}
	+
	\tau \int_{\Omega(t_{n+1}) }\nabla u^{n+1} : \nabla e_{n+1}
	+
	\tau \int_{\Omega(t_{n+1}) } \left((\Vee-\w)(t_{n+1}) \cdot \nabla \right) u^{n+1} \cdot e_{n+1}
	\\
	&- \int_{\Omega(t_{n+1})} u(t_{n+1}) \cdot e_{n+1}
	-
	\tau \int_{\Omega(t_{n+1}) }\nabla u(t_{n+1}) : \nabla e_{n+1}
	\\
	&-
	\tau \int_{\Omega(t_{n+1}) } \left((\Vee-\w)(t_{n+1}) \cdot \nabla \right) u(t_{n+1}) \cdot e_{n+1}
	+\tau \beta \int_{\Omega(t_{n+1}) } e_{n+1}^2.
\end{split}\end{equation}
We use the discrete equation \eqref{eq:DiscreteEquation} with test function $e_{n+1}\in V(t_{n+1})$ to see that
\begin{equation}\begin{split}
	I
	\leq&
	\int_{\Omega(t_{n})} u^{n} \cdot e_{n+1} \circ \Phi_{t_{n+1}}\circ \Phi_{t_n}^{-1}
	+ \tau \int_{\Omega(t_{n+1})} f^{n+1} \cdot e_{n+1}
	- \int_{\Omega(t_{n+1})} u(t_{n+1}) \cdot e_{n+1}
	\\
	&-
	\tau \int_{\Omega(t_{n+1}) }\nabla u(t_{n+1}) : \nabla e_{n+1}
	-
	\tau \int_{\Omega(t_{n+1}) } \left((\Vee-\w)(t_{n+1}) \cdot \nabla \right) u(t_{n+1}) \cdot e_{n+1}
	+\tau \beta \int_{\Omega(t_{n+1}) } e_{n+1}^2.
\end{split}\end{equation}

Using the assumption that the continuous equation \eqref{eq:ContinuousEquationAlternateFormPointwise} holds for every $t\in (0,T)$, we test with $e_{n+1}\in V(t_{n+1})$ to see that
\begin{equation}\begin{split}
	I
	\leq&
	\int_{\Omega(t_{n})} u^{n} \cdot e_{n+1}\circ \Phi_{t_{n+1}}\circ \Phi_{t_n}^{-1}
	+ \tau \int_{\Omega(t_{n+1})} f^{n+1} \cdot e_{n+1}
	- \int_{\Omega(t_{n+1})} u(t_{n+1}) \cdot e_{n+1}
	\\
	&+
	\tau \int_{\Omega(t_{n+1}) }\wmdw u (t_{n+1}) \cdot e_{n+1}
	-
	\tau \int_{\Omega(t_{n+1}) } f^{n+1}\cdot e_{n+1}
	+\tau \beta \int_{\Omega(t_{n+1}) } e_{n+1}^2.
\end{split}\end{equation}
We now collect terms to get
\begin{equation}
	I
	\leq
	\int_{\Omega(t_{n})} u^{n} \cdot e_{n+1} \circ \Phi_{t_{n+1}}\circ \Phi_{t_n}^{-1}
	- \int_{\Omega(t_{n+1})} ( u(t_{n+1})- \tau \wmdw u (t_{n+1}))  \cdot e_{n+1}
	+\tau \beta \int_{\Omega(t_{n+1}) } e_{n+1}^ 2.
\end{equation}

Adding and subtracting the term
$\int_{\Omega_0} (u(t_n) \circ \Phi_{t_n}) \cdot (e_{n+1} \circ \Phi_{t_{n+1}})$
and using the Cauchy--Schwarz inequality and the consistency error bound we get
\begin{equation}
  I \leq
    \|e_n\|_n \|e_{n+1}\|_{n+1}
    + C(u) \tau^{r+1}  \|e_{n+1}\|_{n+1}
    + \tau\beta \|e_{n+1}\|_{n+1}^2.
\end{equation}

From now on we assume that $\tau$ is sufficiently small, such that
$0<\gamma_0^{-1} \leq 1-\tau\beta$ holds true for some fixed $\gamma_0>0$,
independent of $\tau$
and denote $\gamma = (1-\tau\beta)^{-1}\leq \gamma_0$ for convenience.
Then, by subtracting the last term from the previous estimate and inserting $I$
we arrive at
\begin{equation}
  \label{eq:bounded_error}
  (1-\tau\beta)\|e_{n+1}\|_{n+1}^2 + \tau\alpha\|\nabla e_{n+1}\|_{n+1}^2
  \leq (\|e_n\|_n + C(u) \tau^{r+1})  \|e_{n+1}\|_{n+1}
\end{equation}
and by dropping the $\nabla e_{n+1}$ term and dividing by $(1-\tau\beta)\|e_{n+1}\|_{n+1}$ at
\begin{equation}
  \|e_{n+1}\|_{n+1}
  \leq \gamma \|e_n\|_n + \gamma C(u) \tau^{r+1}.
\end{equation}
Using this estimate recursively together with $\|e_0\|_0=0$ yields
\begin{align}
  \begin{split}
  \label{eq:disc_error_l2_bound}
  \|e_n\|_n
  &\leq
    \tau^{r+1}C(u)
    \sum_{k=1}^n \gamma^k
  \leq
    \tau^r C(u)
    (n\tau)
    \exp((n\tau)\beta\gamma)
  =
    \tau^r C(u) t_n \exp(t_n\beta\gamma)\\
  &\leq
    \tau^r C(u) T \exp(T\beta\gamma).
  \end{split}
\end{align}
Here we made use of the elementary estimate
\begin{equation}
  \sum_{k=1}^n \gamma^k
  \leq
    n \gamma^n
  \leq
    n \exp(\gamma-1)^n
  =
    n\exp(n(\gamma-1)).
\end{equation}
and $\gamma-1 = \tau\beta\gamma$.
Inserting \eqref{eq:disc_error_l2_bound} into \eqref{eq:bounded_error} and
bounding $t_n$ and $\gamma$ by $T$ and $\gamma_0$, respectively, we finally get
\begin{equation}
  \|e_n\|_n^2 + \tau \|\nabla e_n\|_n^2
  \leq \tau^{2r}C(u,T,\alpha,\beta,\gamma_0).
\end{equation}

\end{proof}

\begin{corollary}
	Under the assumptions of the previous theorem and assuming $t \mapsto p(t,\cdot)\circ \Phi_t \in C((0,T);L^2_0(\Omega_0))$, it holds that for every $n \geq 1$ with $\tau n <T$,
	\begin{equation}
		\|p^n - p(t_n)\|_{L^2_0} \leq C \tau^{r},
	\end{equation}
	where $p(t_n)$ is the pressure component of the solution to the problem in Corollary \ref{cor:saddle} and $p^n$ is the pressure component to the solution to the discrete problem \eqref{eq:discreteSaddle}.
\end{corollary}
\begin{proof}
Let $e_n^u := u^n -  u(t_n)$ and $e_n^p := p^n - p(t_n)$.
One finds that $(e_n^u,e^p_n) \in H^1_0(\Omega(t_n);\R^d) \times L^2_0(\Omega(t))$ weakly satisfies
\begin{align}
	-\Delta e_n^u + \left( \left(\Vee - \w \right) \cdot \nabla \right) e_n^u + \nabla e^p_n &= \wmdw{u}(t_{n}) - \frac{u^{n} - u^{n-1}}{\tau} && \mbox{ in }\Omega(t_n),
	\\
	\Div e_n^u &= 0 && \mbox{ in } \Omega(t_n),
	\\
	e^u_n|_{\partial \Omega(t_n)} &= 0. &&
\end{align}
Therefore the result follows by standard estimates for the above \emph{stationary} Oseen equation and the estimate in \eqref{eq:assumedError}.
\end{proof}

\section{Conclusion}

In this work, we have shown the well posedness of the Oseen equation on an evolving domain, making use of an evolving space framework.
It turned out that, by considering the problem in suitable evolving spaces of divergence free functions,
it can essentially be treated like a parabolic problem on stationary domain.
With this variational formulation, we have then derived and analysed a first-order time-discretisation.

Regarding further work, uncertainty quantification would be an interesting problem to study,
in particular from the application perspective \cite{tran2019uncertainty, fleeter2020multilevel}.
For example one could model the uncertainty of the initial value, or to consider the equations on random moving domains.
This type of problem was already considered by one of the authors for elliptic PDEs on curved random domain \cite{church2020domain}
and for linear parabolic equations on random domains \cite{Dju18},
where the study of a parabolic Stokes equation on a moving domain was also mentioned.
In this setting, the solution of the equation is a random variable and one is interested in its expected value, for example.
We expect that combining our well-posedness result on moving domains
with the ideas from \cite{Dju18} for random domains
would allow for the treatment of Oseen problems on random domains.
However, this is left for future consideration.

\section*{Acknowledgements}
The authors thank Charles Elliott and Thomas Ranner for many insightful discussions on fluid problems in moving domains.

\appendix
\section{Proof of Lemmas \ref{lem:divPreserving} and \ref{lem:phiBoundedWithInverse} }\label{appendix:ProofOfLemma}

\begin{proof}[Proof of Lemma \ref{lem:divPreserving}]%
	We calculate the divergence of $\phi_t u$,
	\begin{align}%
		\Div (\phi_t u)\nonumber
		=&
		\sum_{i=1}^d \partial_i (\left(D\Phi_t u\right) \circ \Phi_t^{-1} )
		=
		\sum_{i,j=1}^d \partial_i ( (\partial_j (\Phi_t)_i u_j) \circ \Phi_t^{-1})
		\\
		=&\nonumber
		\sum_{i,j,k=1}^d \partial_i (\Phi_t^{-1})_k \partial_k ( \partial_j (\Phi_t)_i u_j) \circ \Phi_t^{-1}
		=
		\sum_{i,j,k=1}^d ((D\Phi_t^{-1})_{ki}\partial_k(\partial_j (\Phi_t)_i u_j) \circ \Phi_t^{-1}
		\\
		=&\nonumber
		\sum_{i,j,k=1}^d ((D\Phi_t^{-1})_{ki} (D\Phi_t)_{ij} \partial_k u_j + (D\Phi_t^{-1})_{ki} \partial_k (D\Phi_t)_{ij} u_j)\circ \Phi_t^{-1}
		\\
		=&
		(\Div u) \circ \Phi_t^{-1} + \sum_{i,j,k=1}^d ((D\Phi_t^{-1})_{ki} \partial_k (D\Phi_t)_{ij} u_j) \circ \Phi_t^{-1}.\label{eq:whereWeUseJacobisFormula}
	\end{align}%
	We now wish to show that the second term of \eqref{eq:whereWeUseJacobisFormula} vanishes.
	Recall Jacobi's formula, which gives the derivative of a determinant of a matrix
	\begin{equation}
		\partial_j \det(D\Phi_t)
		= \det( D\Phi_t) \sum_{i,k=1}^d(D\Phi_t^{-1} )_{ki}\partial_k(D\Phi_t)_{ik}
		= \det( D\Phi_t) \sum_{i,k=1}^d(D\Phi_t^{-1} )_{ki}\partial_k(D\Phi_t)_{ij}.
	\end{equation}
	Moreover, in Lemma \ref{lem:propertiesOfDomainAndDet} we show that $\det(D\Phi_t)$ is constant, therefore the derivative vanishes.
	Hence the second term of \eqref{eq:whereWeUseJacobisFormula} vanishes.
	This shows that $\Div(\phi_t u ) = (\Div u) \circ \Phi_t^{-1}$.

	Now we calculate the divergence of $\phi_{-t} \tilde{u}$,
	\begin{align}
		\Div(\phi_{-t}\tilde{u})
		=&\nonumber
		\sum_{i,j=1}^d \partial_i ((D\Phi_t^{-1})_{ij} \tilde{u}_j \circ \Phi_t)
		\\
		=&\nonumber
		\sum_{i,j=1}^d \left( \partial_i (D\Phi_t^{-1})_{ij} \tilde{u}_j \circ \Phi_t + \sum_{k=1}^d(D\Phi^{-1}_t)_{ij}\partial_i (\Phi_t)_k \partial_k \tilde{u}_j \circ \Phi_t \right)
		\\
		=&\label{eq:whereWeUseJacobisFormula2}
		\Div(\tilde{u})\circ \Phi_t + \sum_{i,j=1}^d\partial_i (D\Phi_t^{-1})_{ij} \tilde{u}_j \circ \Phi_t.
	\end{align}
	We again must deal with the extra term, the second term of \eqref{eq:whereWeUseJacobisFormula2}.
	It is known that the derivative of the inverse of a matrix is given by $\partial_i(D\Phi_t^{-1}) = -(D\Phi_t)^{-1} \partial_i D\Phi_t (D\Phi_t)^{-1}$,
	therefore
	\begin{equation}
		\sum_{i=1}^d \partial_i\partial_j (\Phi_t)_i
		=
		\sum_{i,k,l=1}^d (D\Phi_t^{-1})_{ik}\partial_i (D\Phi_t)_{kl} (D\Phi_t^{-1})_{lj},
	\end{equation}
	where we again note that, by Jacobi's formula and $\det(D\Phi_t^{-1}) = 1$, $\sum_{i,k=1}^d(D\Phi_t^{-1})_{ik}\partial_i (D\Phi_t)_{kl}=0$.
	This has shown $\Div(\phi_{-t} \tilde{u}) = \Div(\tilde{u}) \circ \Phi_t$ and completed the result.
\end{proof}

\begin{proof}[Proof of Lemma \ref{lem:phiBoundedWithInverse}]
Let $u \in L^2(\Omega(t);\R^d)$, we calculate
		\begin{equation}
			\int_{\Omega(t)} |\phi_t u|^2
			=
			\int_{\Omega(t)} |D\Phi_t u|^2 \circ \Phi^{-1}_t
			=
			\int_{\Omega_0} |D\Phi_t u|^2 {\det(D\Phi_t)},
		\end{equation}
		where we recall $\det(D\Phi_t) = 1$, therefore,
		\begin{equation}
			\|\phi_t u \|_{H(t)} \leq \|D\Phi_t\|_{L^\infty} \|u\|_{\HO}.
		\end{equation}
		For $\tilde{u}\in L^2(\Omega(t);\R^d)$, the following, almost identical calculation,
		\begin{equation}
			\int_{\Omega_0} |\phi_{-t} \tilde{u}|^2
			=
			\int_{\Omega(t)} |D\Phi_t^{-1} \circ \Phi_t^{-1} \tilde{u}|^2
			\leq
			\|D\Phi_t^{-1}\|_{L^\infty}^2\|\tilde{u}\|^2_{H(t)}.
		\end{equation}
		Now let $u \in H^1_0(\Omega_0;\R^d)$, we calculate the Dirichlet energy,
		\begin{equation}
			\int_{\Omega(t)} |D \left( \phi_t u\right)|^2
			=
			\int_{\Omega(t)} |D \left( D\Phi_t \circ \Phi_t^{-1} u \circ \Phi_t^{-1}\right)|^2
			=
			\int_{\Omega_0} |D \left( D\Phi_t \circ \Phi_t^{-1} u \circ \Phi_t^{-1}\right)|^2 \circ \Phi_t.
		\end{equation}
		We now look at the integrand termwise,
		\begin{equation}\begin{split}
		\partial_i \left( D\Phi_t \circ \Phi_t^{-1} u \circ \Phi_t^{-1} \right)_j \circ \Phi_t
		=&
		\sum_{k=1}^d \partial_i \left( \left(\partial_k (\Phi_t)_j u_k \right)\circ \Phi_t^{-1}\right)\circ \Phi_t
		\\
		=&
		\sum_{k,l=1}^d \left( \partial_i (\Phi^{-1}_t)_l \circ \Phi_t \partial_l \left(\partial_k(\Phi_t)_ju_k\right) \right)
		\\
		\leq&
		\|D(\Phi_t^{-1})\|_{L^\infty} \left( \|D^2\Phi_t\|_{L^\infty} |u| + \|D\Phi_t\|_{L^\infty} |Du|\right)
		\end{split}\end{equation}
		Hence
		\begin{equation}
			\int_{\Omega(t)} |D(\phi_t u)|^2
			\leq
			C\|D(\Phi_t^{-1})\|_{L^\infty}^2\left( \|D^2\Phi_t\|_{L^\infty}^2 \|u\|_{\HO}^2 + \|D\Phi_t\|_{L^\infty}^2\|Du\|_{\HO}^2 \right),
		\end{equation}
		therefore,
		\begin{equation}
			\|\phi_t u \|_{V(t)}^2
			\leq
			C \|D(\Phi_t^{-1})\|_{L^\infty}^2  \left(\|D^2\Phi_t\|_{L^\infty}^2 \|u\|_{\HO}^2 +\|D\Phi_t\|_{L^\infty}^2 \|u\|_{\VO}^2\right).
		\end{equation}
		For $\tilde{u} \in H^1(\Omega(t);\R^d)$,
		\begin{equation}
			\int_{\Omega_0} |D\left( \phi_{-t} \tilde{u}) \right)|^2
			=
			\int_{\Omega_0} | D \left( D\Phi^{-1}_t \tilde{u} \circ \Phi_t\right)|^2
			=
			\int_{\Omega(t)} | D \left( D\Phi^{-1}_t \tilde{u} \circ \Phi_t\right)|^2 \circ \Phi_t^{-1}.
		\end{equation}
		Again, looking at the integrand termiwse,
		\begin{equation}\begin{split}
			\partial_i \left( \left(D\Phi_t^{-1} \right)_{jk} \tilde{u}_k \circ \Phi_t \right) \circ \Phi_t^{-1}
			=&
			\sum_{k=1}^d \partial_i \left( \left( D(\Phi_t^{-1})_{jk}  \tilde{u}_k \right)\circ \Phi_t \right)\circ \Phi_t^{-1}
			\\
			=&
			\sum_{k,l=1}^d\partial_i (\Phi_t)_l \partial_l ( \partial_k(\Phi_t^{-1})_j \tilde{u}_k)
			\\
			\leq&
			\|D\Phi_t\|_{L^\infty} \left( \|D^2(\Phi_t^{-1})\|_{L^\infty} |\tilde{u}| + \|D(\Phi_t^{-1})\|_{L^\infty} |D\tilde{u}|\right).
		\end{split}\end{equation}
		As before, this gives
		\begin{equation}\begin{split}
			\|\phi_{-t} \tilde{u}\|_{\VO}
			\leq
			C \|D\Phi_t\|_{L^\infty}^2 \left( \|D^2(\Phi_t^{-1})\|_{L^\infty}^2 \|\tilde{u}\|_{H(t)}^2 + \|D(\Phi^{-1}_t)\|_{L^\infty}^2 \|\tilde{u}\|_{V(t)}^2\right).
		\end{split}\end{equation}
\end{proof}

\end{document}